\documentclass[12pt,a4paper,twoside]{article}

\usepackage[applemac]{inputenc}
\usepackage[T1]{fontenc}
\usepackage[english]{babel}

\usepackage{fancyhdr}
\usepackage{indentfirst} 
\usepackage{newlfont}
\usepackage{verbatim}
\usepackage{enumerate}
\usepackage{hyperref}
\usepackage{pdflscape}

\usepackage{amsthm}  
\usepackage{amssymb}  
\usepackage{amscd}
\usepackage{amsmath} 
\usepackage{latexsym} 
\usepackage{amsfonts}
\usepackage{amsbsy}   
\usepackage{mathrsfs} 
\usepackage{graphicx}
\usepackage{color}

\theoremstyle{plain}
\newtheorem{teo}{Theorem}
\newtheorem{prop}[teo]{Proposition}

\theoremstyle{definition}

\newtheorem{ese}[teo]{Example}

\theoremstyle{remark}
\newtheorem{oss}[teo]{Remark}

\newcommand{\rp}[1]{\ensuremath{\mathbb{RP}^{#1}}}
\newcommand{\s}[1]{\ensuremath{\mathbf{S}^{#1}}}

\begin{document}

\title{Equivalence of two diagram representations of links in lens spaces and essential invariants
\footnote{Work performed under the auspices of G.N.S.A.G.A.
of C.N.R. of Italy and supported by M.U.R.S.T., by the University
of Bologna, funds for selected research topics.}}

\author{Alessia Cattabriga, Enrico Manfredi, Lorenzo Rigolli}

\maketitle

\begin{abstract}

In this paper we study the relation between two diagrammatic representations of links in lens spaces: the disk diagram  introduced in \cite{CMM}
and the grid diagram introduced in \cite{BGH,Co} and we find how to pass from one to the other.   We also investigate whether 
the HOMFLY-PT invariant and  the Link Floer Homology are essential invariants,  that is, we try to understand if these invariants are able to distinguish 
links in $L(p,q)$  covered by the same link in $\s3$. In order to do so, we generalize the combinatorial definition of Knot Floer Homology 
in lens spaces developed   in \cite{BGH,MOS} 
to the case of links and we analyze how both the invariants change when we switch the orientation of the link.  
\\ {{\it Mathematics Subject
Classification 2010:} Primary  57R58, 57M27; Secondary  57M25.}\\

\end{abstract}

\begin{section}{Introduction}

For many years the study of knots and links has been confined to the case of $\s3$, where different combinatorial 
representations  as well as  powerful invariants were developed in order to study the equivalence 
problem. In the last ten years, as far as the knowledge  on  3-manifolds was  improving, knot theory has shifted also to manifolds different from $\s3$. 
In this setting, lens spaces play a leading role for many different reasons. For example some knots conjectures in $\s3$
can be rephrased in terms of links in lens spaces, as, for example, the Berge conjecture  (see \cite{Berge, Berge2, Gr}).
Furthermore there are interesting articles explaining applications of knots in lens spaces outside mathematics: \cite{Ste} 
exploits them to describe topological string theories and \cite{BM} uses them to describe the resolution of a biological DNA recombination problem. 
Another fundamental reason is that, among three manifolds, 
lens spaces are quite well understood. They are defined as finite cyclic quotient  of $\s3$, but they admit
many different (combinatorial) representation that have been extended to represent also the links contained inside them.  
In \cite{La,LR1,LR2} Dehn surgery representation of lens spaces is used to construct mixed link diagram, while in \cite{CM}  the representation of lens spaces 
as genus one Heegaard splitting  leads to an algebraic representation of  links in lens spaces 
using the elements of the mapping class group. The same representation of lens spaces is used in \cite{BGH} to generalize to links in lens spaces
the notion of grid diagram  introduced in \cite{Bru,Cro,Din} for the 3-sphere case, and used in \cite{MOS} to describe a combinatorial version of 
the Link Floer Homology. Exploiting this representation, the authors manage to extend  Knot Floer Homology
to lens space, whereas in  \cite{Co} a HOMFLY-PT invariant is constructed. 
A disk diagram representation as well as Reidemeister type moves are introduced in \cite{Dr,CMM}  looking at lens spaces as the result of pasting a 3-ball along its 
boundary. Using this diagram,  in \cite{Dr, Mr2} a Jones type polynomial and a HOMFLY and Kauffman skein modules are constructed for the case of 
$L(2,1)=\mathbb{RP}^3$. This diagram is generalized to all lens spaces in \cite{CMM}, where the authors use it to compute the 
fundamental group as well as the twisted Alexander polynomial.   
As far as so many invariant have been extended to 
links in lens spaces, a natural question arising is the following: which of them is able to distinguish different links
in a certain lens space covered by the same link in $\s3$? Such an invariant is called \textit{essential}. In \cite{M, M2} the author finds many examples of different links in the
same lens space covered by the same link in $\s3$ and discuss the essentiality of some geometric invariants as the twisted Alexander polynomial. In this paper
we analyze the case of the HOMFLY-PT invariant and the Link Floer Homology. In order to do so, we describe how to pass from a  grid diagram representation to a  
disk diagram representation of the same link.  

This paper is organized as follows.
In section 2 we recall the definition of disk diagram and the corresponding Reidemeister type moves introduced in \cite{CMM}. 

In Section 3, first we resume the definition of grid diagram introduced in \cite{BGH,Co}, then we find 
how to pass from a disk diagram of a given link $L$ in $L(p,q)$ to a grid diagram of the same link and vice versa. We  also discuss the correspondence
between Reidemeister type moves on the disk diagram and equivalence moves on the grid diagram. 

In Section 4  we deal with the HOMFLY-PT invariant of links in lens spaces introduced in \cite{Co}. We study how it behaves under change of orientation
of the link and we compute it on some examples in order to discuss whether this invariant is essential or not.

Finally, Section 5 concerns Link Floer Homology. We  generalize the combinatorial definition of Link Floer Homology, developed  in  \cite{MOS} for links in $\s3$
and in  \cite{BGH}  for knots in lens spaces,
to the case of links in lens spaces and study its behaviour under change of orientation. We find examples of links with the same covering distinguished by this
invariant. All the detailed computations of the Link Floer Homology of such examples are contained  in the Appendix.

The results stated in this paper hold both in the \emph{Diff} category and in the \emph{PL} category, as well as in the \emph{Top} 
category if we consider only tame links. Moreover we consider oriented links up to ambient isotopy.

\end{section}


\begin{section}{Links in lens spaces via disk diagrams}

In this section we recall the notion of the disk diagram for  links in lens spaces developed in  \cite{CMM}, and the corresponding equivalence moves.

\paragraph{A model for lens spaces}
We start by recalling the model that we use for lens spaces. Let $p$ and $q$ be two coprime integers such that $ 0 \leqslant q < p$. 
The unit ball  is the set  $B^3=\{(x_{1},x_{2},x_{3}) \in \mathbb{R}^{3} \ | \ x_{1}^{2}+x_{2}^{2}+x_{3}^{2}\leqslant1\}$ and $E_{+}$ and $E_{-}$ denote, respectively, the upper and the lower closed hemisphere of $\partial B^{3}$.  Label with  $B^{2}_{0}$ the equatorial disk, that is  the intersection of the plane $x_{3}=0$ with $B^{3}$.  Finally let $N=(0,0,1)$ and $S=(0,0,-1)$.
Consider the rotation  \mbox{$g_{p,q}\colon E_{+} \rightarrow E_{+}$} of $2 \pi q /p$ radians around the $x_{3}$-axis and the reflection  \hbox{$f_{3}\colon E_{+} \rightarrow E_{-}$} with respect to the plane $x_{3}=0$ (see  Figure~\ref{L(p,q)}).

\begin{figure}[h!]                      
\begin{center}                         
\includegraphics[width=9cm]{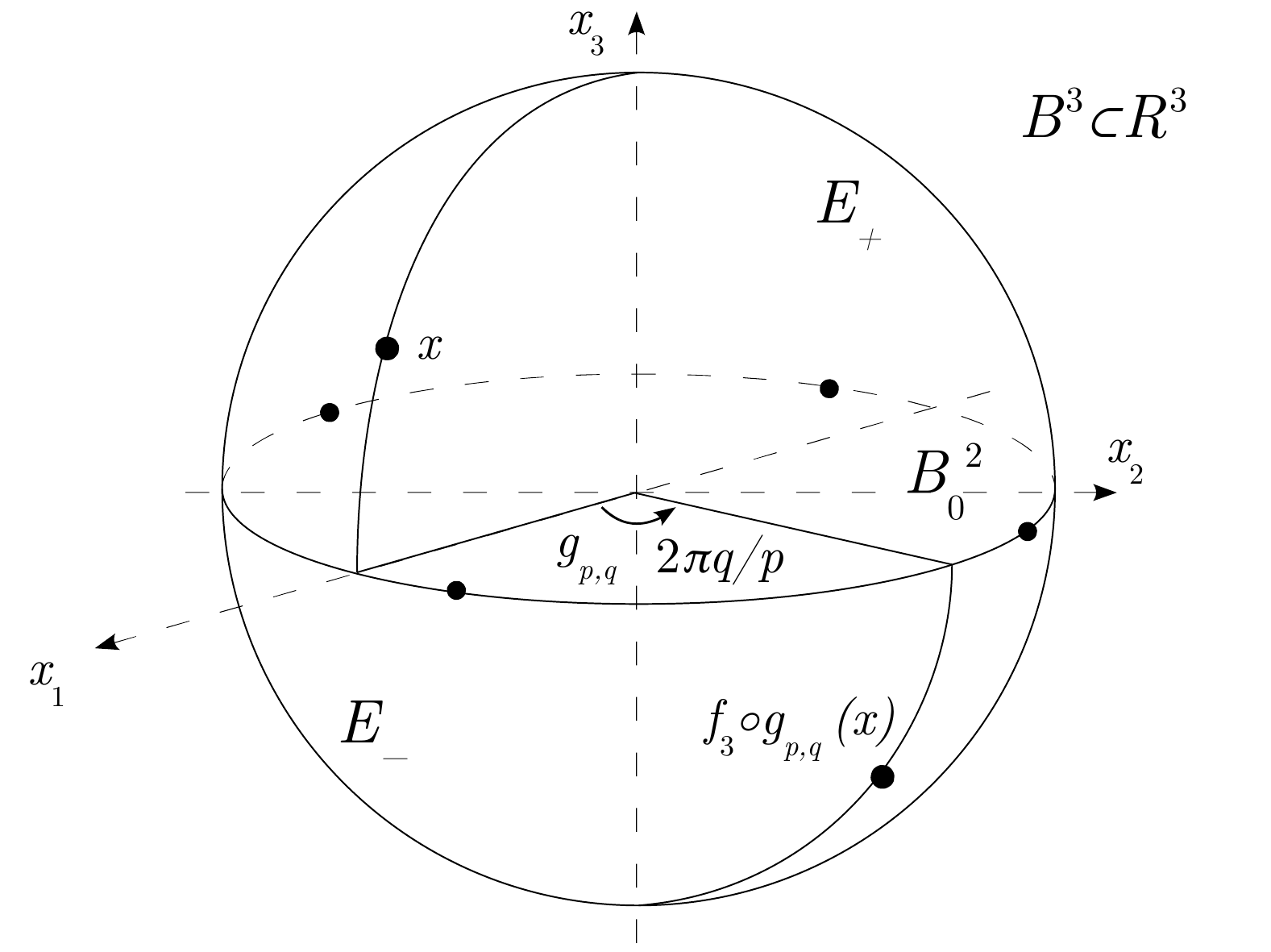}
\caption[legenda elenco figure]{A model for $L(p,q)$.}\label{L(p,q)}
\end{center}
\end{figure}

The \emph{lens space} $L(p,q)$ is the quotient of $B^{3}$ by the equivalence relation on $\partial B^{3}$ which identifies $x \in E_{+}$ with $f_{3} \circ g_{p,q} (x) \in E_{-}$. We denote with \mbox{$F\colon B^{3} \rightarrow L(p,q)=B^{3} / \sim$} the quotient map. Notice that  on the equator $\partial B^{2}_{0}=E_{+} \cap E_{-}$ each equivalence class contains $p$ points. Clearly we have $L(1,0)\cong \s{3}$ and $L(2,1) \cong \rp{3}$.

\paragraph{The construction of the disk diagram}

 We  briefly recall the construction of the disk diagram for a link in a lens space  developed in  \cite{CMM}. Throughout the section  we assume $p>1$.
Let $L\subset L(p,q)$ be a link and consider $L'=F^{-1}(L)$. By moving $L$ via a small isotopy in $L(p,q)$, we can suppose that $L'$ is the disjoint union of closed curves in $\text{int}(B^{3})$ and arcs properly embedded in $ B^{3}$  not containing $N$ and $S$.
Let $\mathbf{p} \colon B^{3} \smallsetminus \{ N,S \} \rightarrow B^{2}_{0}$ be the projection defined by $\mathbf{p}(x)=c(x) \cap B^{2}_{0}$, where $c(x)$ is the circle (possibly a line) through $N$, $x$ and $S$.
Project $L'$ using $\mathbf{p}_{|_{L'}}\colon L' \rightarrow B^{2}_{0}$. 

As in the classical case,  we can assume, by moving $L$ via a small isotopy, that the projection $\mathbf{p}_{|_{L'}} \colon L' \rightarrow B^{2}_{0}$ of $L$ is \emph{regular}, namely
\begin{enumerate} \itemsep-4pt
\item[(1)] the projection of $L'$ contains no cusps;
\item[(2)] all auto-intersections of $\mathbf{p}(L')$ are transversal;
\item[(3)] the set of multiple points is finite, and all of them are actually double points;
\item[(4)] no double point is on $\partial B^{2}_{0}$.
\end{enumerate}

Finally, double points are resolved with underpasses and overpasses as in the diagram for links in $\s3$.
A \emph{disk diagram} of a link $L$ in $L(p,q)$ is a regular projection of $L'=F^{-1}(L)$ on the equatorial disk $B^{2}_{0}$, with specified overpasses and underpasses (see Figure~\ref{link3}).

Notice that an orientation of the link $L$ induces an orientation on $L'$ and so on a diagram of $L$. 

\begin{figure}[h!]                      
\begin{center}                         
\includegraphics[width=12cm]{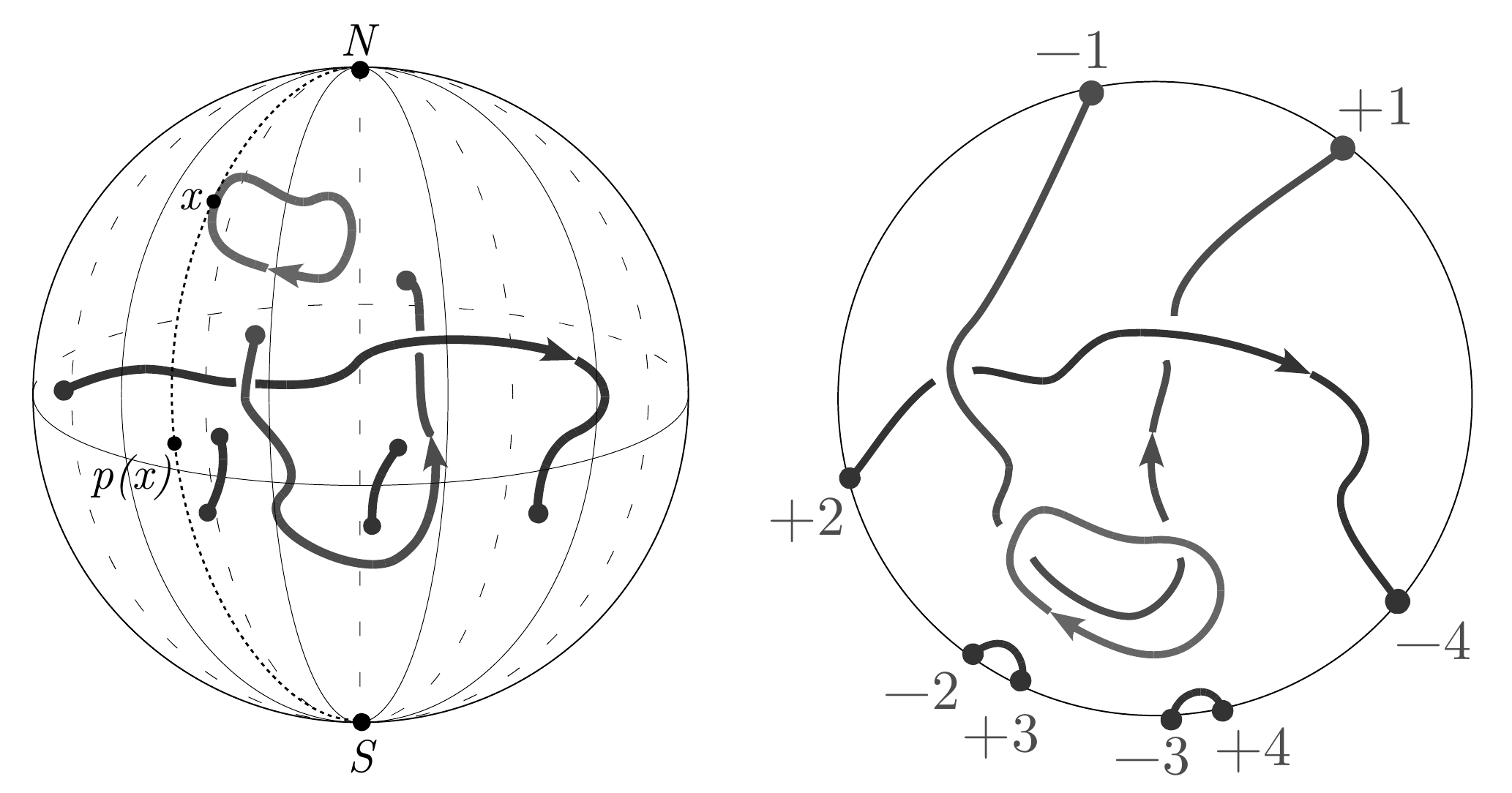}
\caption[legenda elenco figure]{A link in $L(9,1)$ and its corresponding disk diagram.}\label{link3}
\end{center}
\end{figure}

In order to make the disk diagram more comprehensible, we add an indexation of the boundary points of the projection as follows: first, assume that the equator $\partial B^{2}_{0}$ is oriented counterclockwise if we look at it from $N$, then, according to this orientation, label with $+1, \ldots, +t$ the endpoints of the projection of the link coming from the upper hemisphere, and with $-1, \ldots, -t$ the endpoints coming from the lower hemisphere, respecting the rule $+i \sim -i$. An example is shown in Figure~\ref{link3}.

\paragraph{Reidemeister type moves}
In \cite{CMM} it is shown that two disk diagrams represent the same link if and only if they are connected by a finite sequence of the seven Reidemeister type moves depicted in Figure \ref{R17}.

\begin{figure}[h!]                      
\begin{center}                         
\includegraphics[width=14cm]{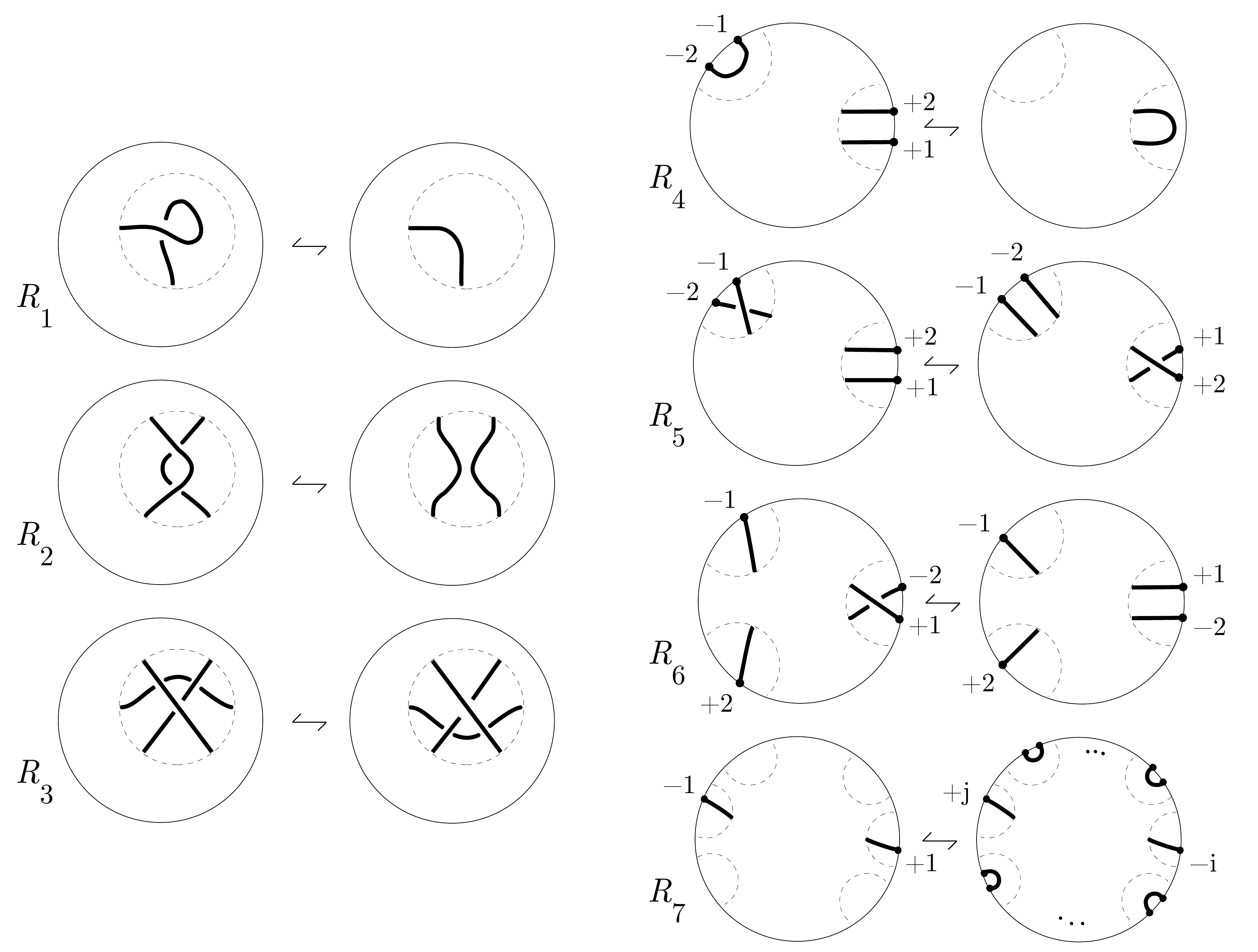}
\caption[legenda elenco figure]{Generalized Reidemeister moves.}\label{R17}
\end{center}
\end{figure}

\end{section}

\begin{section}{Connection with the grid diagram of links in lens space}

In this section first we recall the notion of grid the diagram for  links in  lens spaces developed in \cite{BGH} and \cite{Co}, 
then  we explain how to transform a disk diagram into a grid diagram and vice versa, showing also the connection between the equivalence  moves on the two different diagrams.

\paragraph{Grid diagram of links in lens space}
  A \emph{(toroidal) grid diagram} $G$ in $L(p,q)$  with grid number $n$  is a quintuple
  $(T^{2}, \boldsymbol{\alpha}, \boldsymbol{ \beta},\mathbb{O}, \mathbb{X})$ that satisfies the following conditions 
  (see Figure \ref{grid} for an example with grid number $3$ in $L(4,1)$)

  \begin{figure}[h!]                      
\begin{center}                         
\includegraphics[width=15cm]{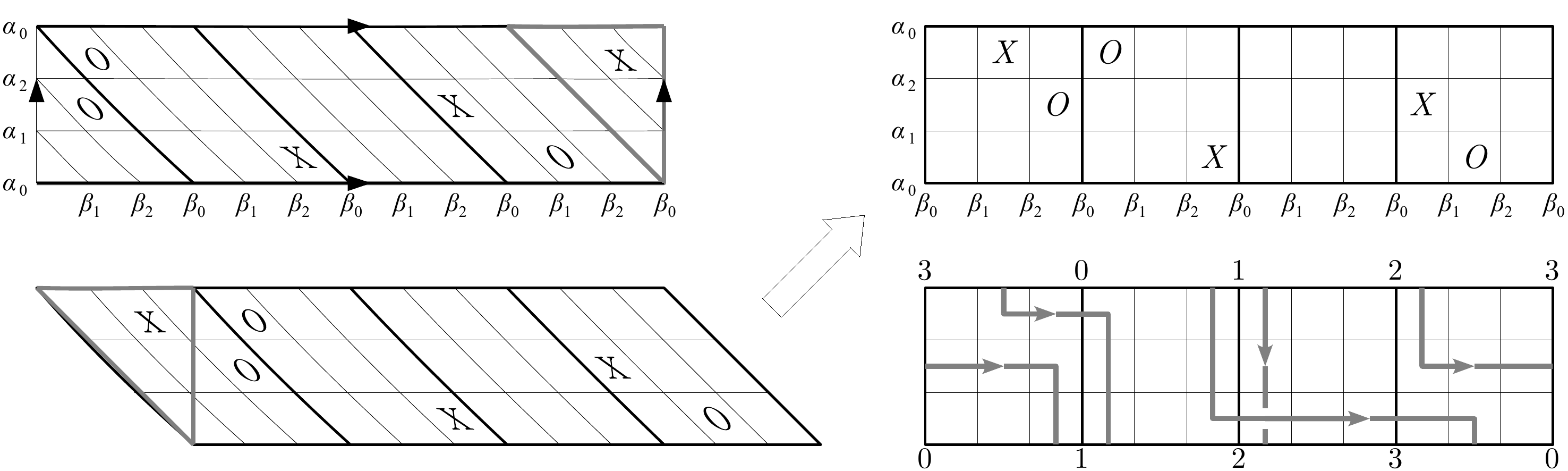}
\caption[legenda elenco figure]{From a grid diagram with  grid number $3$  to its corresponding link in $L(4,1)$.}\label{grid}
\end{center}
\end{figure}
  
\begin{itemize}
\item $T^{2}$ is the standard oriented torus $\mathbb{R}^{2}/ \mathbb{Z}^{2}$, where $\mathbb{Z}^{2}$ is the lattice generated by the vectors $(1,0)$ and $(0,1)$;

\item $\boldsymbol{\alpha}= \{ \alpha_{0}, \ldots , \alpha_{n-1} \}$ are the images in $T^{2}$ of the $n$ lines in $\mathbb{R}^{2}$ 
described by the equations $y=i/{n}$, for $i=0,\ldots,n-1$; the complement 
$T^{2} \smallsetminus (\alpha_{0} \cup \ldots \cup \alpha_{n-1})$ has $n$ connected annular components, called the \emph{rows} of the grid diagram; 

\item $\boldsymbol{\beta}= \{ \beta_{0}, \ldots , \beta_{n-1} \}$ are the images in $T^{2}$ of the $n$ lines in $\mathbb{R}^{2}$ described by the equations
$y=-\frac{p}{q}(x-\frac{i}{pn})$, for $i=0,\ldots,n-1$; the complement \hbox{$T^{2} \smallsetminus (\beta_{0} \cup \ldots \cup \beta_{n-1})$} has $n$ connected annular components, 
called the \emph{columns} of the  grid diagram; 

\item $\mathbb{O}=\{ O_{0}, \ldots, O_{n-1} \}$ (resp. $\mathbb{X}=\{ X_{0}, \ldots, X_{n-1} \}$) are $n$ 
points in \hbox{$T^{2} \smallsetminus (\boldsymbol{\alpha} \cup \boldsymbol{\beta})$} called \emph{markings}, 
such that any two points in $\mathbb{O}$ (resp.   $\mathbb{X}$)  lie in different rows and in different columns.
\end{itemize}

In order to make the identifications of the diagram's boundary easier to understand, it is possible to perform the  ``shift'' depicted 
in Figure \ref{grid}. Notice that, if we forget about $L(p,q)$'s identifications, the curve $\beta_0$ divides the rectangle of  a grid diagram into 
$p$ adjacent squares, that we will call  \textit{boxes} of the diagram.

A grid diagram $G$ represents an oriented link $L\subset L(p,q)$ obtained as follows. First, denote with $V_{\boldsymbol{\alpha}}$ and $V_{\boldsymbol{\beta}}$  the two solid 
tori having, respectively, $\boldsymbol{\alpha}$ and  $\boldsymbol{\beta}$  as   meridians. Clearly     $V_{\boldsymbol{\alpha}} \cup_{T^{2}} V_{\boldsymbol{\beta}}$
is a genus one Heegaard splitting representing $L(p,q)$.  Then connect  
\begin{itemize}
 \item[(1)] each $X_{i}$ to the unique $O_{j}$ lying in the same row with an arc embedded in the row and disjoint from the curves of $\boldsymbol{\alpha}$, and 
\item[(2)] each $O_{j}$ to the unique $X_{l}$ lying in the same column by an arc embedded in the column and disjoint from the curves of $\boldsymbol{\beta}$, 
\end{itemize} obtaining a multicurve immersed in $T^{2}$. Finally   remove the self-intersections,  pushing the lines of (1)  into $V_{\boldsymbol{\alpha}}$
and  the lines of (2)  into $V_{\boldsymbol{\beta}}$.   
The orientation on $L$ is obtained by orienting the horizontal arcs connecting the markings   from the $X$ to the $O$.
See Figure \ref{grid} for an example in $L(4,1)$.

Notice that, the presence in the grid diagram  of an  pair of marking $X$ and $O$ in the same position  corresponds to a trivial component of the represented link
(see the bottom row of the first box of  Figure \ref{trivial}).

By Theorem 4.3 of \cite{BGH}, each link  $L \subset L(p,q)$ can be represented by a grid diagram. 
The idea of the proof is a PL-approximation with orthogonal lines of the link  projection on the torus.

\paragraph{Equivalence moves for  grid diagrams}

A \emph{grid (de)stabilization} is a move that (decreases) increases by one the grid number. Figure \ref{stab} shows an example in $L(5,2)$ 
of a $X:NW$ grid (de)stabilization, where $X$ is the grid marking chosen for the stabilization and $NW$ refers to the arrangement of the new markings. 
Of course, we can have also (de)stabilization with respect to $O$ markings and with $NE,SW$ and $SE$ arrangements.

\begin{figure}[h!]                      
\begin{center}                         
\includegraphics[width=12cm]{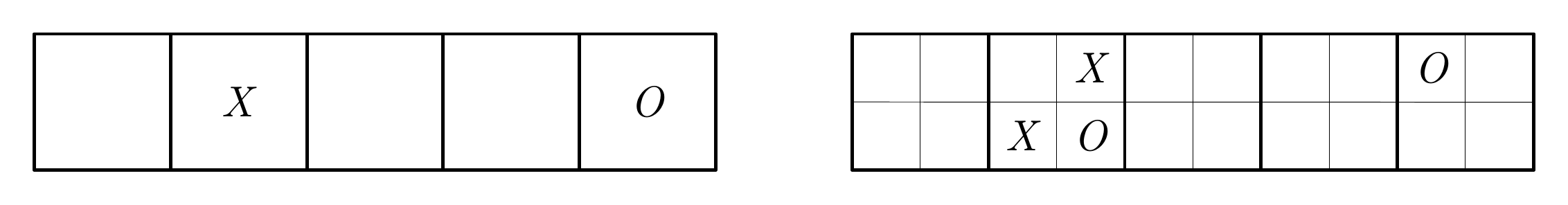}
\caption[legenda elenco figure]{An example of (de)stabilization in $L(5,2)$.}\label{stab}
\end{center}
\end{figure}

A grid diagram \emph{commutation} interchanges either two adjacent columns or two adjacent  rows as follows.
Let $A$ be the annulus containing the two considered columns (or rows) $c_1$ and $c_2$. 
The annulus is divided into $pn$ parts by the rows (columns). Let $s_1$ and $s_2$ be the two
bands of the annulus containing the markings of $c_1$. 
Then the commutation is \emph{interleaving} 
if the markings of $c_2$ are in different components of $A-s_1-s_2$, and \emph{non-interleaving} otherwise (see Figure \ref{comm}).

\begin{figure}[h!]                      
\begin{center}                         
\includegraphics[width=14cm]{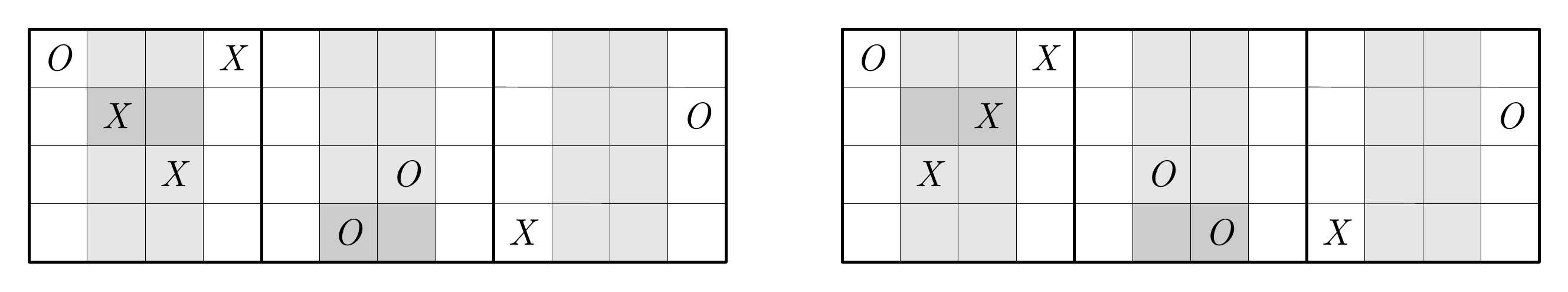}
\caption[legenda elenco figure]{An example of non-interleaving commutation in $L(3,1)$.}\label{comm}
\end{center}
\end{figure}

\begin{prop}{\upshape\cite{BG}}
Two grid diagrams of links in  $L(p,q)$ represent  the same link if and only if there exists a finite sequence of (de)stabilizations and non-interleaving commutations connecting the two grid diagrams.
\end{prop}

Please notice that there are also two other  hidden moves on a grid diagram, depending directly on the projection of the link on the Heegaard torus:
we can make a \emph{cyclic permutation} of the rows or of the columns -following the pasting of the torus- and we 
can do a \emph{reverse connection} by connecting the grid markings also in the opposite direction.

\paragraph{Passing from disk diagrams to twisted grid diagrams and vice versa}

The following two propositions describe how to pass from a disk diagram to a grid diagram representing the same link and vice versa.
\begin{prop}\label{GLp}
Let $L$ be a link in $L(p,q)$ assigned via a grid diagram $G_{L}$. Then we can obtain the disk diagram $D_{L}$ representing $L$
in  the following way (see Figure \ref{GL})

\begin{figure}[h!]                      
\begin{center}                         
\includegraphics[width=14cm]{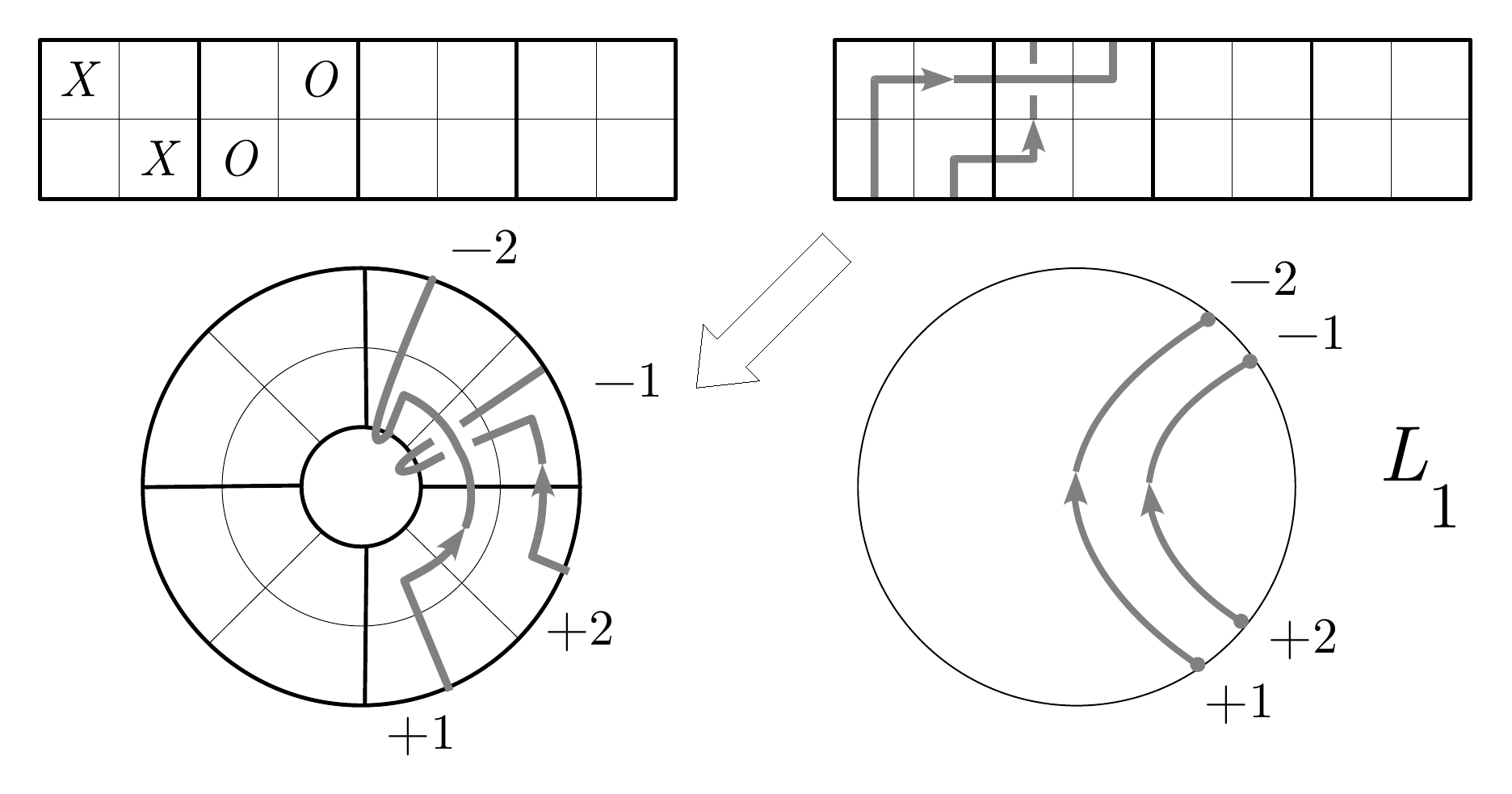}
\caption[legenda elenco figure]{From grid diagram $G_{L}$ to disk diagram $D_{L}$ in $L(4,1)$.}\label{GL}
\end{center}
\end{figure}
\begin{itemize}
 \item  consider the grid diagram $G_{L}$ and draw the link according to the previous convention;
 \item round the rectangle into a circular annuli, joining the first and the last column: 
 the horizontal lines become circles and the vertical lines become radial lines on the disk diagram. 
 \item the lower boundary points on the rectangle become plus type boundary points on the disk. 
 The upper boundary points, instead, are inside the disk: by
 moving them under all the circle lines  we can bring them on the boundary of the disk, so that they become minus-type boundary points.
\end{itemize}
\end{prop}

\begin{proof}
The grid diagram of a link in a lens space comes from the representation of lens spaces as Heegaard splitting. 
That is to say, our grid diagram is the toric Heegaard surface. 
If we want to transform the grid diagram into the disk diagram $D_{L}$ we have to put our Heegaard surface inside $B^3$ in the model of $L(p,q)$ where we quotient $B^3$ by the relation $\sim$ on its boundary. This can be done as Figure \ref{GLdim} shows.

\begin{figure}[h!]                      
\begin{center}                         
\includegraphics[width=14cm]{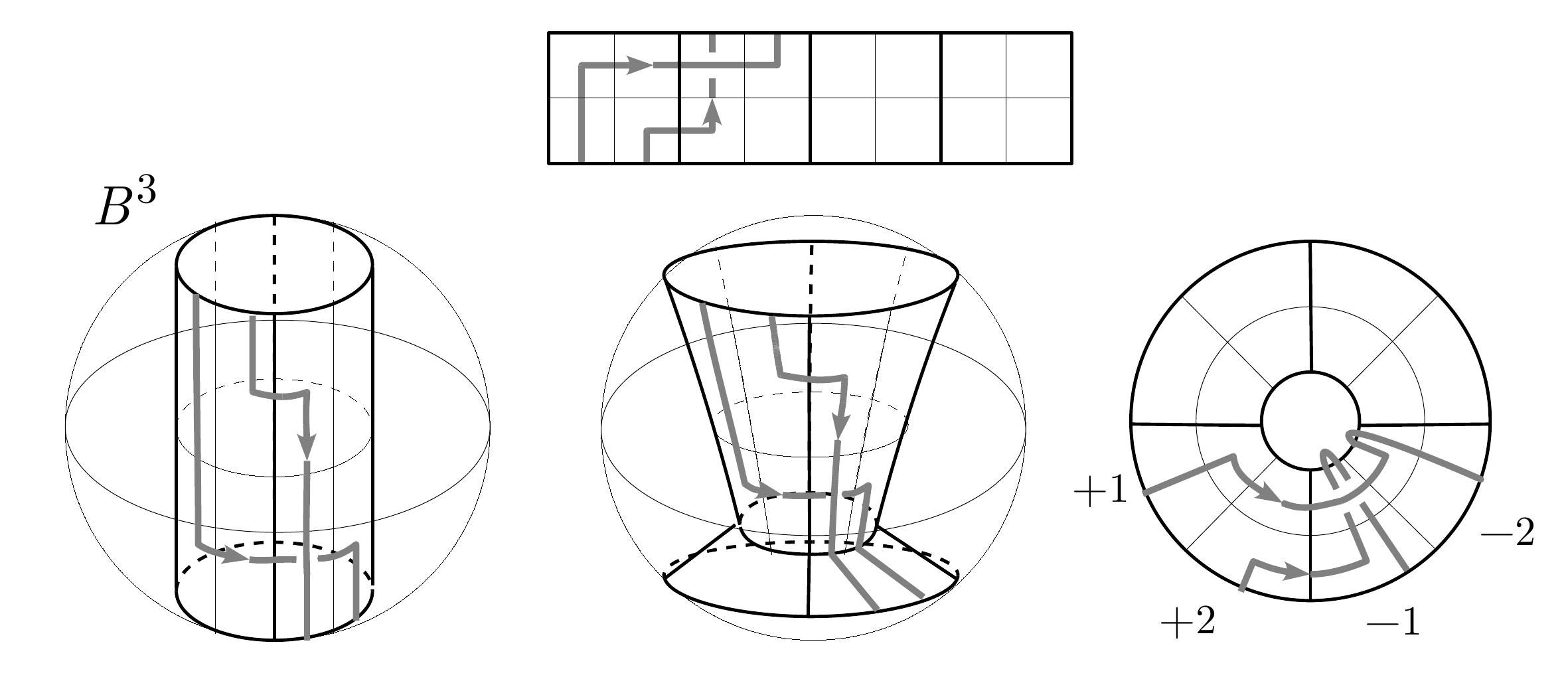}
\caption[legenda elenco figure]{How to insert the grid diagram of $L$ into  the $B^3$-model of $L(4,1)$.}\label{GLdim}
\end{center}
\end{figure}

Now we want to project this surface on the equatorial disk $B^2_0$, and, in order to have a regular projection of the link, we deform the Heegaard torus as in Figure \ref{GLdim}. The projection of the deformed grid diagram on $B^2_0$ gives $D_{L}$.
\end{proof}

\begin{oss}
If the grid diagram $G_{L}$ has grid number $n$, then  the disk diagram $D_{L}$, obtained from $G_L$, has at most $n (p-1)$ boundary points. Indeed,
the number of boundary points of $D_{L}$ is exactly the number of the points onto the lower and upper boundary of the rectangle of $G_{L}$, that is, 
at most, $n(p-1)$.
\end{oss}

In the opposite direction, when we know the disk diagram $D_{L}$ of a link $L\subset L(p,q)$, how can we recover the grid diagram $G_{L}$?

\begin{prop}
\label{DG}
Let $L$ be a link in $L(p,q)$, defined by a disk diagram $D_{L}$, then we can get a grid diagram $G_{L}$ of $L$ as follows (see 
Figure \ref{LG})

\begin{figure}[h!]                      
\begin{center}                         
\includegraphics[width=12cm]{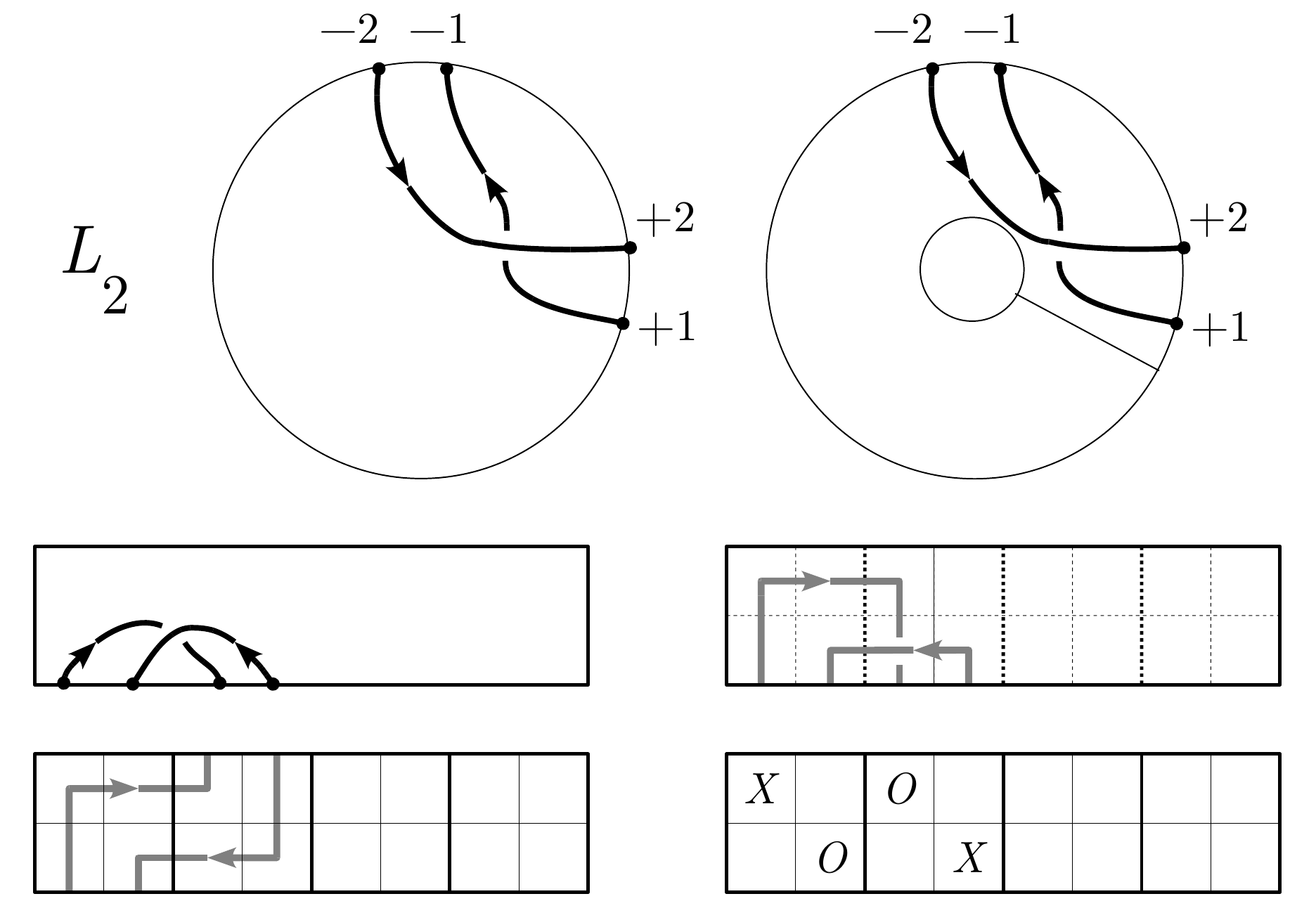}
\caption[legenda elenco figure]{From disk diagram $D_{L}$ to grid diagram $G_{L}$ in $L(4,1)$.}\label{LG}
\end{center}
\end{figure}
\begin{itemize}
 \item consider the disk diagram $D_{L}$ and cut the disk along a ray between the  $+1$ point and the previous boundary
 point (according to the orientation of the disk), obtaining a rectangle;
 \item make an orthogonal PL-approximation of the link's arcs,
 putting all the crossings with horizontal overpass and vertical underpass;
 \item shift the boundary endpoint of $-1, \ldots, -t$ from the lower to the upper side of the rectangle, passing under all the lines;
 \item  put $X$ and $O$ markings on the square corners of the link projection.
\end{itemize}

\end{prop}

\begin{proof}
It is  exactly the converse of the proof of Proposition \ref{GL}. The only difference is that here we have to use the orthogonal PL-approximation suggested by Theorem 4.3 of \cite{BGH}.
\end{proof}

Using  Propositions \ref{GLp}  and \ref{DG}, it is possible to find also a correspondence between the Reidemeister moves on the disk diagrams (depicted in Figure \ref{R17}) 
and the grid diagram's equivalence
moves described in the previous paragraph. This correspondence is summed up in Table \ref{moves}.

\begin{table}[h!]
\label{moves}
\begin{center}
\begin{tabular}{|c|c|}
\hline 
Disk diagram & Grid diagram  \\
\hline
$R_1$ & (de)stab. \\
\hline
$R_2$ & non-inter. comm.  \\
\hline
$R_3$ & non-inter. comm.   \\
\hline
$R_4$ & cyclic perm. of rows  \\
\hline
$R_5$ & cyclic perm. of rows  \\
\hline
$R_6$ & non-inter. comm.   \\
\hline
$R_7$ & column reverse connection  \\
\hline
\end{tabular}
\end{center}
\label{table}
\end{table}

\end{section}

\begin{section}{Essential invariants: the HOMFLY-PT polynomial}

In this section we deal with the HOMFLY-PT polynomial developed in \cite{Co} in order to understand if it is an essential invariant, that is if it is able to 
distinguish links covered by the same link in $\s3$. 
We start by recalling its definition (see \cite{Co} for the details).\\

We say that a link in $L(p,q)$ is  \textit{trivial} if it can be represented by a grid diagram satisfying the following conditions
\begin{itemize}
 \item the markings in each box lie only on the principal diagonal (the one going from NE-corner to the SW-corner);
 \item all the $O$-markings are contained  in the the first box  (from the left); 
 \item the $X$-markings in the same box are contiguous, and if the first box contains $X$-markings,
 one of them lies in the SW corner;
 \item for each $X$-marking, all the other $X$-markings lying in a row below, must lie in a column on  the left.
\end{itemize}

A trivial link will be denoted as $U_{i_0,i_{p-1},\ldots, i_{1}}$ where $i_j\in\mathbb N$ is the number of components of the link belonging to the 
$j$-th homology class. In Figure \ref{trivial} is depicted the trivial link $U_{1,0,1,2}\subset L(4,1)$ having one $0$-homologous component,
zero $1$-homologous component, one $2$-homologous component and two $3$-homologous components.
\begin{figure}[h!]                      
\begin{center}                      
\includegraphics[width=12cm]{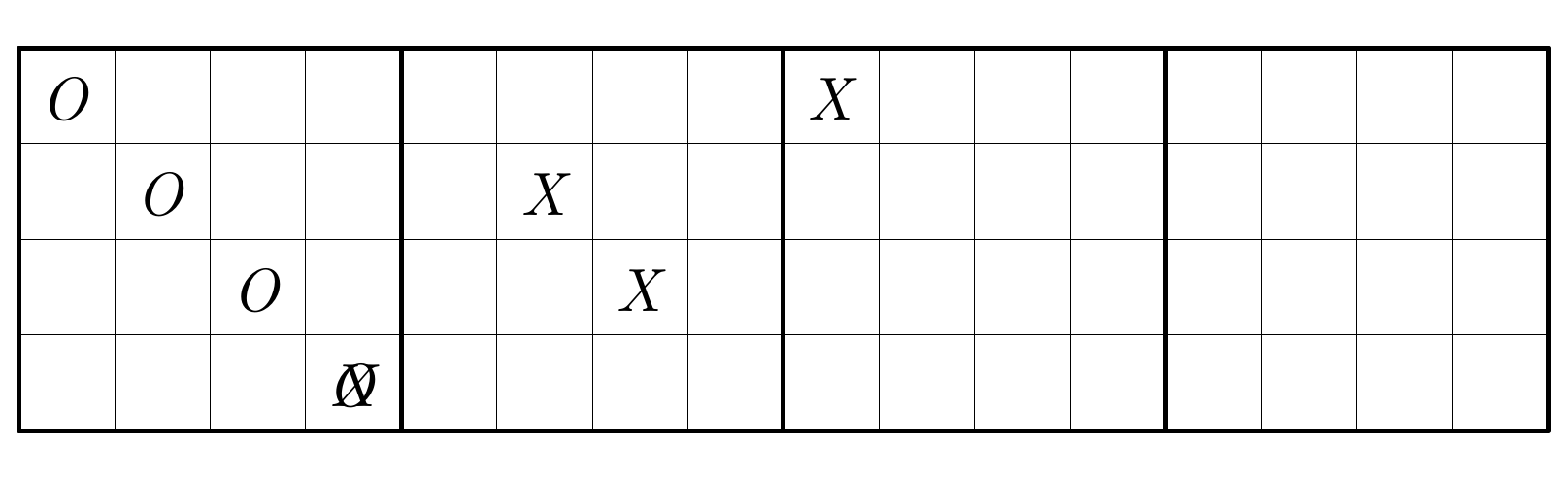}
\caption[legenda elenco figure]{Grid diagram for the trivial link $U_{1,0,1,2}$ in $L(4,1)$.}\label{trivial}
\end{center}
\end{figure}

\begin{teo}{\upshape{\cite{Co}}}\label{Corn}
Let $\mathcal{L}$ be the set of isotopy classes of links in $L(p, q)$ and let $\mathcal{TL} \subset \mathcal{L}$
denote the set of isotopy classes of trivial links. Define $\mathcal{TL}^{\ast} \subset \mathcal{TL}$
to be those trivial links with no nullhomologous components. Let $U$ be the isotopy class of the standard unknot, a local knot in $L(p, q)$ 
that bounds an embedded disk. Given a value $J_{p,q}(T) \in \mathbb{Z}[a^{\pm1},z^{\pm1}]$ for every $T \in \mathcal{TL}^{\ast} $, there is a unique map $J_{p,q} \colon \mathcal{L} \rightarrow \mathbb{Z}[a^{\pm1}, z^{\pm1}]$ such that
\begin{itemize}
\item $J_{p,q}$ satisfies the skein relation $a^{-p}J_{p,q}(L_{+}) - a^{p}J_{p,q}(L_{-}) = z J_{p,q}(L_{0})$.
\item $J_{p,q}(U) = \big( \frac{a^{-1}-a}{z}\big)^{p-1}$
\item $J_{p,q} (U \sqcup L) = \big( \frac{a^{-p} -a^{p}}{z}\big)  J_{p,q} (L) $
\end{itemize}
 \end{teo}
 
As usual, the links $L_{+}$,$L_{-}$, and $L_0$ differ only in a small neighborhood of a double point: Figure \ref{skein} shows how this difference appears on grid diagrams.

\begin{figure}[h!]                      
\begin{center}                         
\includegraphics[width=7cm]{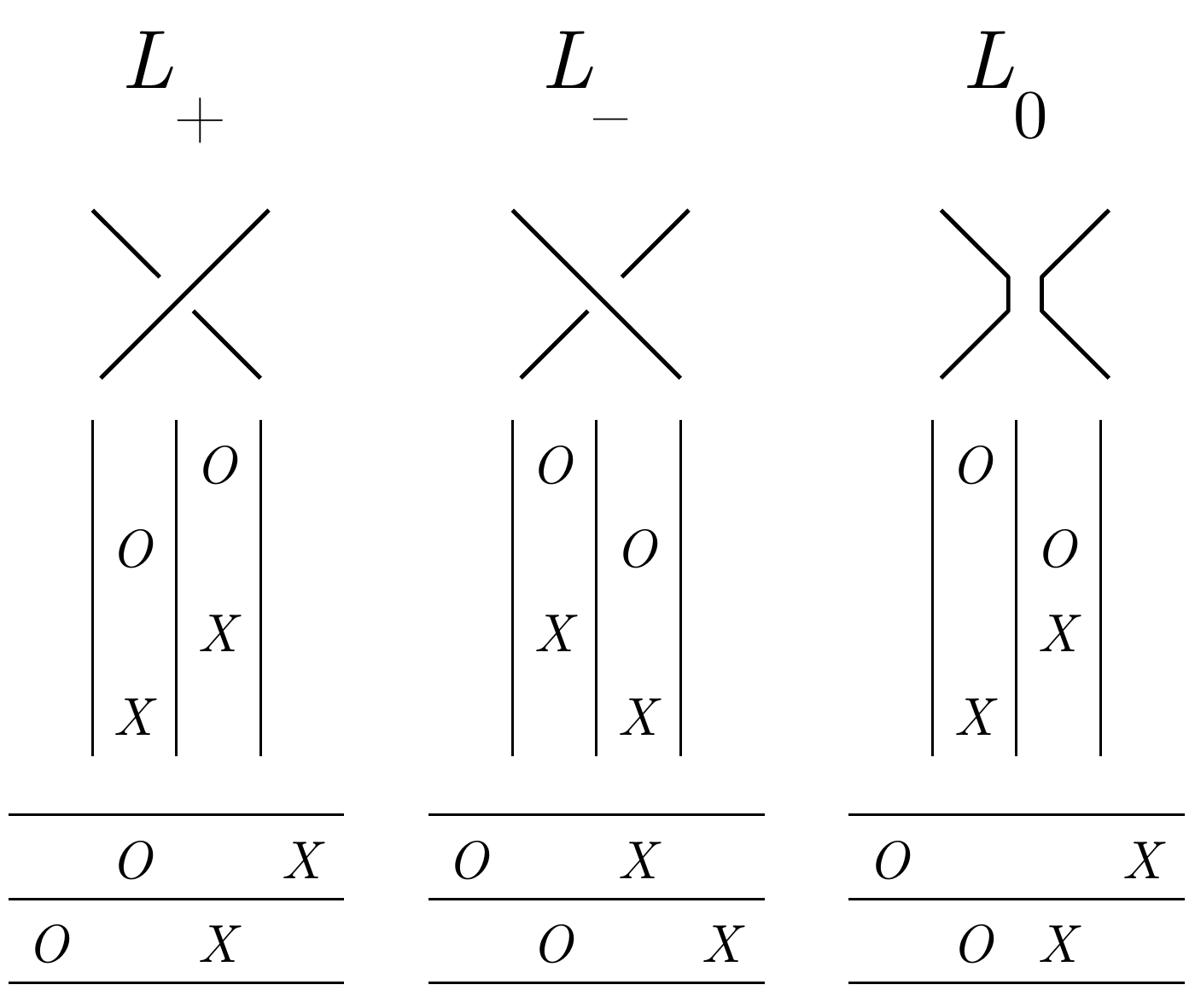}
\caption[legenda elenco figure]{Grid skein relation.}\label{skein}
\end{center}
\end{figure}

The HOMLFY-PT invariant produced by Theorem \ref{Corn} is not yet a polynomial, Cornwell suggests to produce a polynomial in the usual HOMFLY two variables by defining $J_{p,q}$ on the trivial links as the classic HOMFLY-PT polynomial of their lift in the $3$-sphere.
Clearly, the essentiality of the HOMFLY-PT invariant  depends on the  assignment of a value to $J_{p,q}$  on the class
$\mathcal{TL}^*$: an assignment based on the lift makes the invariant much less sensitive in this direction.

\paragraph{Behavior under change of orientation}

What happens to the HOMFLY-PT invariant  when  we change the orientation of every component of the link? In the case of $\s3$,
the  classic HOMFLY-PT polynomial does not change, but, in  $L(p,q)$ things are different since $L(p,q)$ is homologically non-trivial. 

\begin{prop}\label{HomOrCh}
Let $L$ be a link in $L(p,q)$ and denote with $-L$ the link obtained by reversing the orientation of  each component. 
If the HOMFLY-PT invariant of $L$ can be written as $J_{p,q}(L)= \sum a^{k}z^{h} J_{p,q}(U_{i_0,i_{p-1},i_{p-2},\ldots, i_1})$, 
then $J_{p,q}(-L)=\sum a^{k}z^{h} J_{p,q}(U_{i_0,i_1,\ldots,i_{p-2},i_{p-1}})$.
\end{prop}

\begin{figure}[h!]
\begin{center}
\includegraphics[width=12cm]{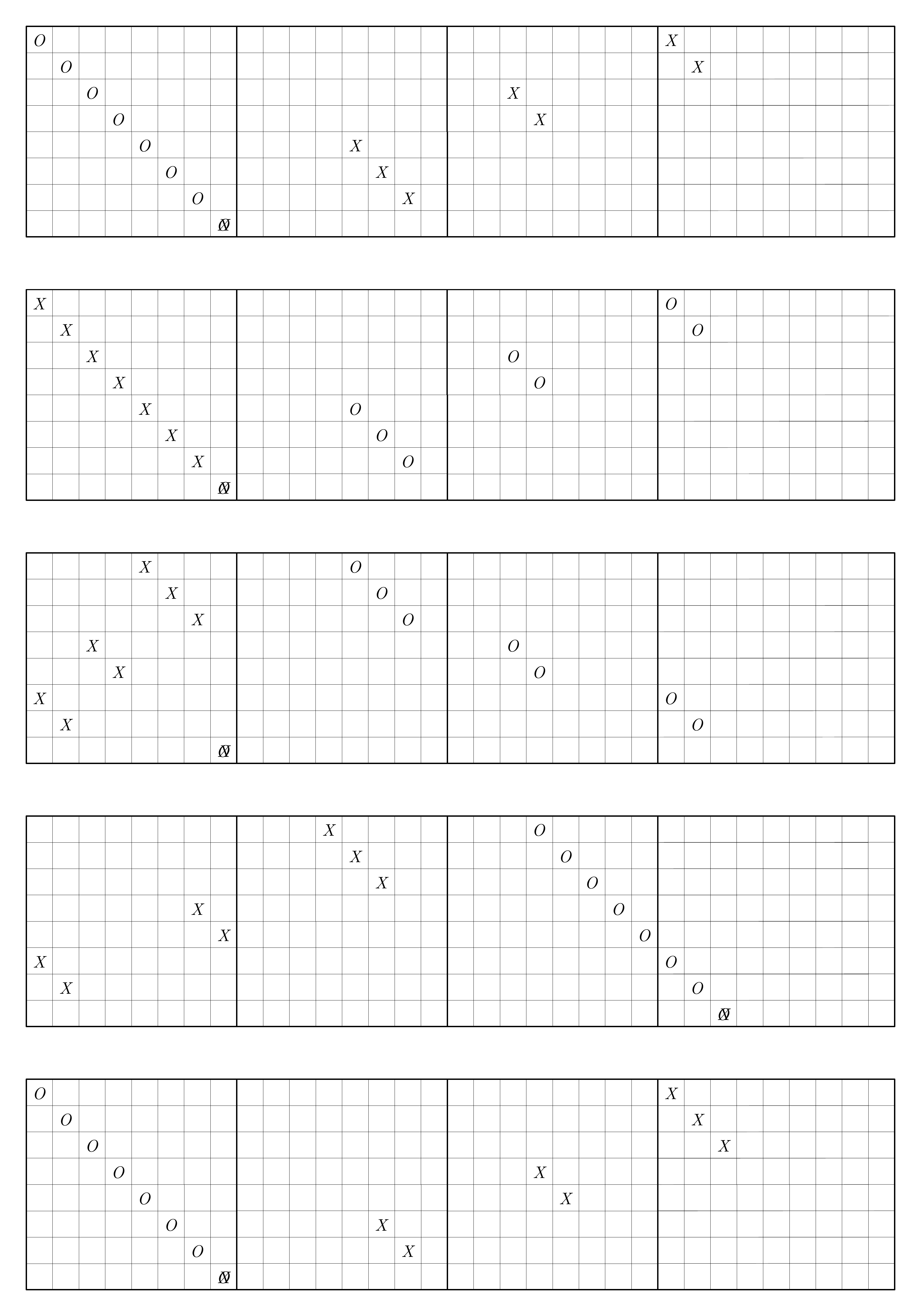}
\caption[legenda elenco figure]{Reduction to trivial link of   $-U_{1,2,2,3}$ in $L(4,1)$.}\label{orrev}
\end{center}
\end{figure}

\begin{proof}
As for the HOMFLY-PT polynomial for links in the 3-sphere, the skein reduction of both $L$ and $-L$ is the same,
because if we change the orientation in $L_{+}$, $L_{-}$ and $L_{0}$ we still get respectively $L_{+}$, $L_{-}$ and $L_{0}$.
But if we change the orientation in the trivial links, then we find a different trivial link; more precisely, looking at
Figure \ref{orrev}, if we change the orientation on the trivial link $U_{i_0,i_{p-1},i_{p-2},\ldots, i_1}$, 
and perform  at first a sequence of non-interleaving row commutations, then, 
a sequence of non-interleaving column commutations and finally some cyclic permutation of columns
we obtain the trivial link $U_{i_0,i_1,\ldots,i_{p-2},i_{p-1}}$.
\end{proof}

Usually,  in $L(p,q)$, the links $L$ and $-L$ are    non equivalent (since they are generally homologically different). So, the
last proposition suggests a way to construct examples of non-equivalent oriented links with the same lifting in $\s3$, distinguished by the 
HOMFLY-PT invariant. Indeed it is enough to find a link $L$ lifting to an invertible link and such that $L$ is non isotopic to $-L$. For example, the knots $K$ 
and $-K$ in $L(3,1)$ in Figure 
\ref{fig_banale}
are different since the first one is 1-homologous whereas the second one is 2-homologous, but they both lift to the trivial knot in $\s3$.

\begin{figure}[h!]
\begin{center}
\includegraphics[width=7cm]{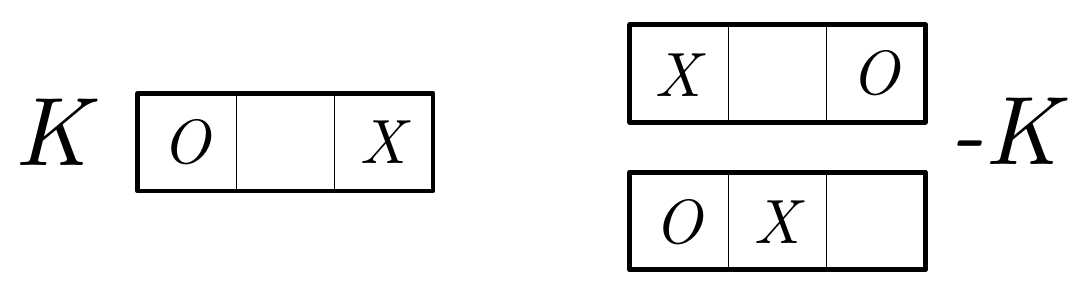}
\caption[legenda elenco figure]{Knots $K$ and $-K$ in $L(3,1)$ both lifting to the trivial knot in $\s3$.}\label{fig_banale}
\end{center}
\end{figure}

But what does it  happen if the links with the same  lift don't differ only from an orientation change? In  \cite{M} the author finds  many examples of different links
in $L(p,q)$ with the same covering in $\s3$. We end the section by  computing the HOMFLY-PT invariant of some of them. The first two examples are 
quite simple, since they are pairs of different trivial links: having the same HOMFLY-PT invariant or not depends on how we define 
 $J_{p,q}$ on $\mathcal{TL}^{\ast}$. On the contrary, in the third example,  that is  much more complicated, the two links are distinguished by the HOMFLY-PT polynomial.

\begin{ese}
\label{perlorenzo}
The two knots  of Figure \ref{CE1grid} are $K_1$ and $K_2$ in $L(5,2)$. They are different since $K_1$ is 1-homologous, 
while $K_2$ is 2-homologous, but they both lift to the trivial knot in $\s3$ (see \cite{M}).
Using Proposition \ref{DG}, we get \hbox{$K_1=U_{0,0,0,0,1}$} and $K_2=U_{0, 0, 0,1,0}$ in $L(5,2)$.  So, if we assume 
\hbox{$J_{p,q}(L):=J_{1,0}(\widetilde{L})$} on  trivial links, we clearly have
$J_{p,q} (K_{1})=1=J_{p,q} (K_{2}).$ It is possible to  generalize this example to $L \left (p, \frac{p \pm 1}{2} \right )$ (see \cite{M}).
\begin{figure}[h!]                      
\begin{center}                         
\includegraphics[width=8cm]{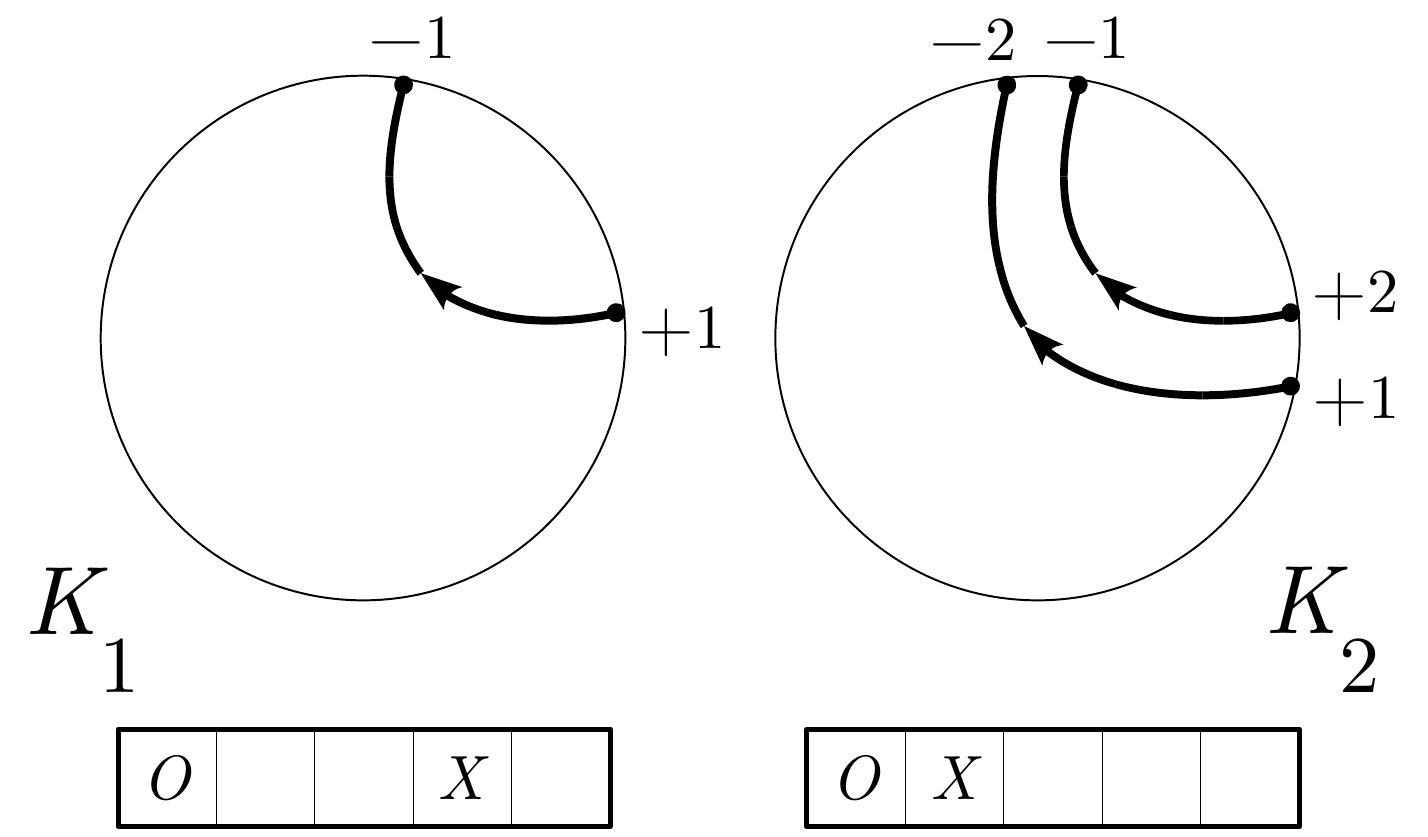}
\caption[legenda elenco figure]{Diagrams for different knots in $L(5,2)$ with trivial lift.}\label{CE1grid}
\end{center}
\end{figure}
\end{ese}

\begin{ese}
\label{perlorenzo2}
The two links $L_A,L_B\subset L(4,1)$ represented in Figure \ref{CE2grid} are non-equivalent  since the first one 
is a knot, whereas the second one is a two component link.
Nevertheless, they both lift to the Hopf link  in $\s3$ (see \cite{M}).
Transforming the disk diagram into a grid diagram (see Proposition \ref{DG}) and 
performing  some destabilizations and non-interleaving commutations, we see that they are nothing else than the trivial links 
$L_A=U_{0,0,1,0}$ and $L_B=U_{0,1,0,1}$.   So, if we assume 
$J_{p,q}(L):=J_{1,0}(\widetilde{L})$ on  trivial links, we clearly have
$J_{4,1} (L_{A})=a z + a z^{-1}-a^{3}z^{-1}=J_{4,1} (L_{B}).$
 \begin{figure}[h!]                      
\begin{center}                         
\includegraphics[width=10cm]{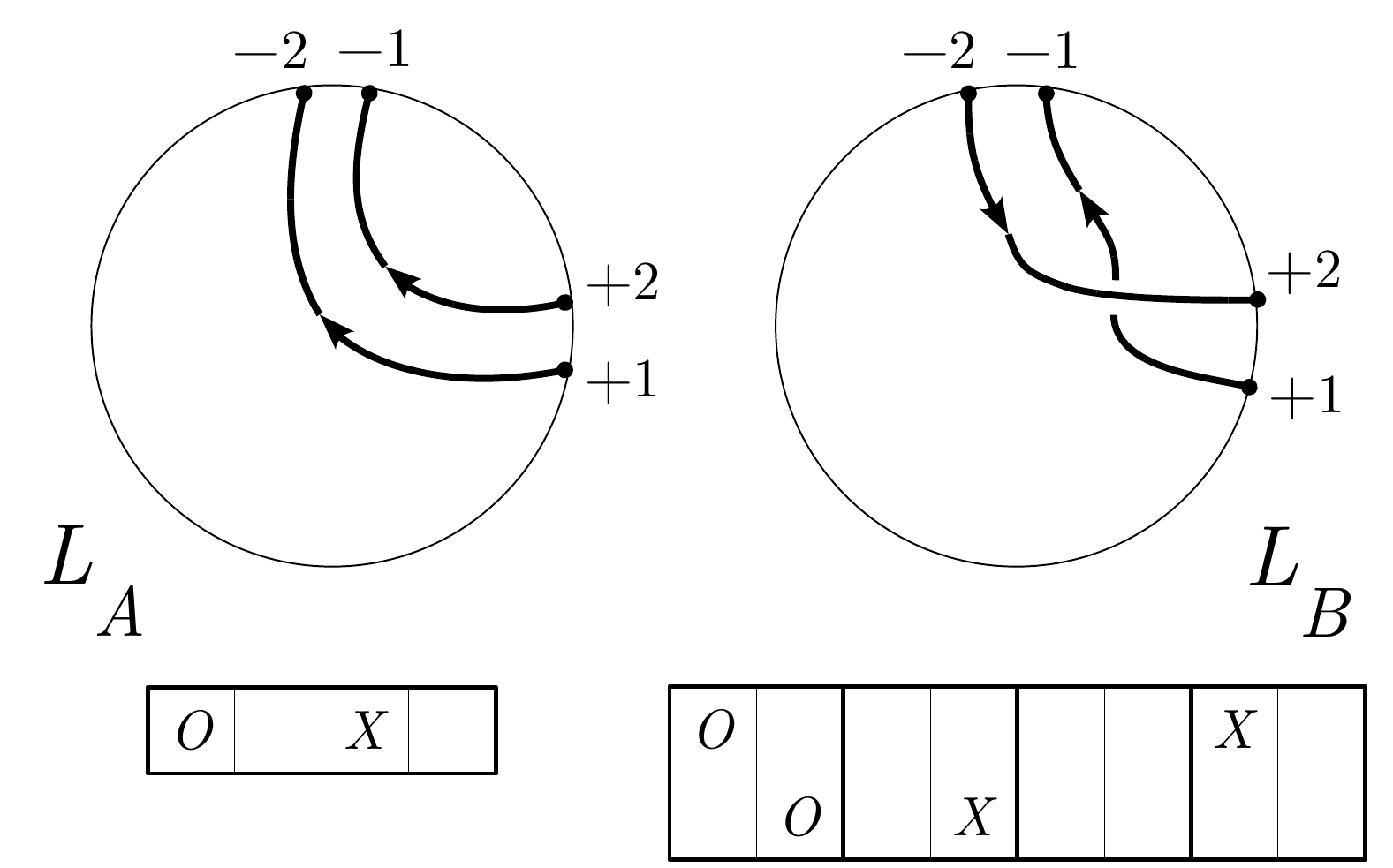}
\caption[legenda elenco figure]{Diagrams for different links in $L(4,1)$ with Hopf link lift.}\label{CE2grid}
\end{center}
\end{figure}
\end{ese}

\begin{ese}
The two  links $A_{2,2}$ and $B_{2,2}$ in $L(4,1)$ depicted in Figure \ref{CE3grid} are non equivalent, having different Alexander polynomial,
but  they both lift to the Hopf link in $\s3$ (see \cite{M}). 
The computation of their   HOMFLY-PT invariant is very long.
The skein reduction tree is quite big, so we report here only the final result
\begin{eqnarray*}
J_{4,1} (A_{2,2})&=&
(a^{24}+3 a^{24} z^2+ a^{24} z^4) J_{4,1} (U_{0,0,2,0}) + \\
&\ &+ (3 a^{28} z+4 a^{28} z^3+ a^{28} z^5) J_{4,1} (U_{1,0,0,0}) + \\
&\ &+ (3 a^{24} z^2 + 4a^{24}  z^4+a^{24} z^6) J_{4,1} (U_{0,1,0,1}) \\
 J_{4,1} (B_{2,2})&
= & (a^{24}+2a^{24}z^2+a^{24}z^4) J_{4,1} (U_{0,0,2,0})+\\
&\ & +(a^{28}z+2a^{28}z^3+a^{28}z^5)J_{4,1} (U_{1,0,0,0})+\\
&\ & +(a^{24}z^2+2a^{24}z^4+a^{24}z^6)J_{4,1} (U_{0,1,0,1})+\\
&\ &+(a^{20}z+2a^{20}z^3)J_{4,1} (U_{0,2,1,0})+\\
&\ &+ a^{20}  z J_{4,1} (U_{0,0,1,2})+ a^{24}  z^2 J_{4,1} (U_{0,2,0,2}). 
\end{eqnarray*}

\begin{figure}[h!]                      
\begin{center}                         
\includegraphics[width=9cm]{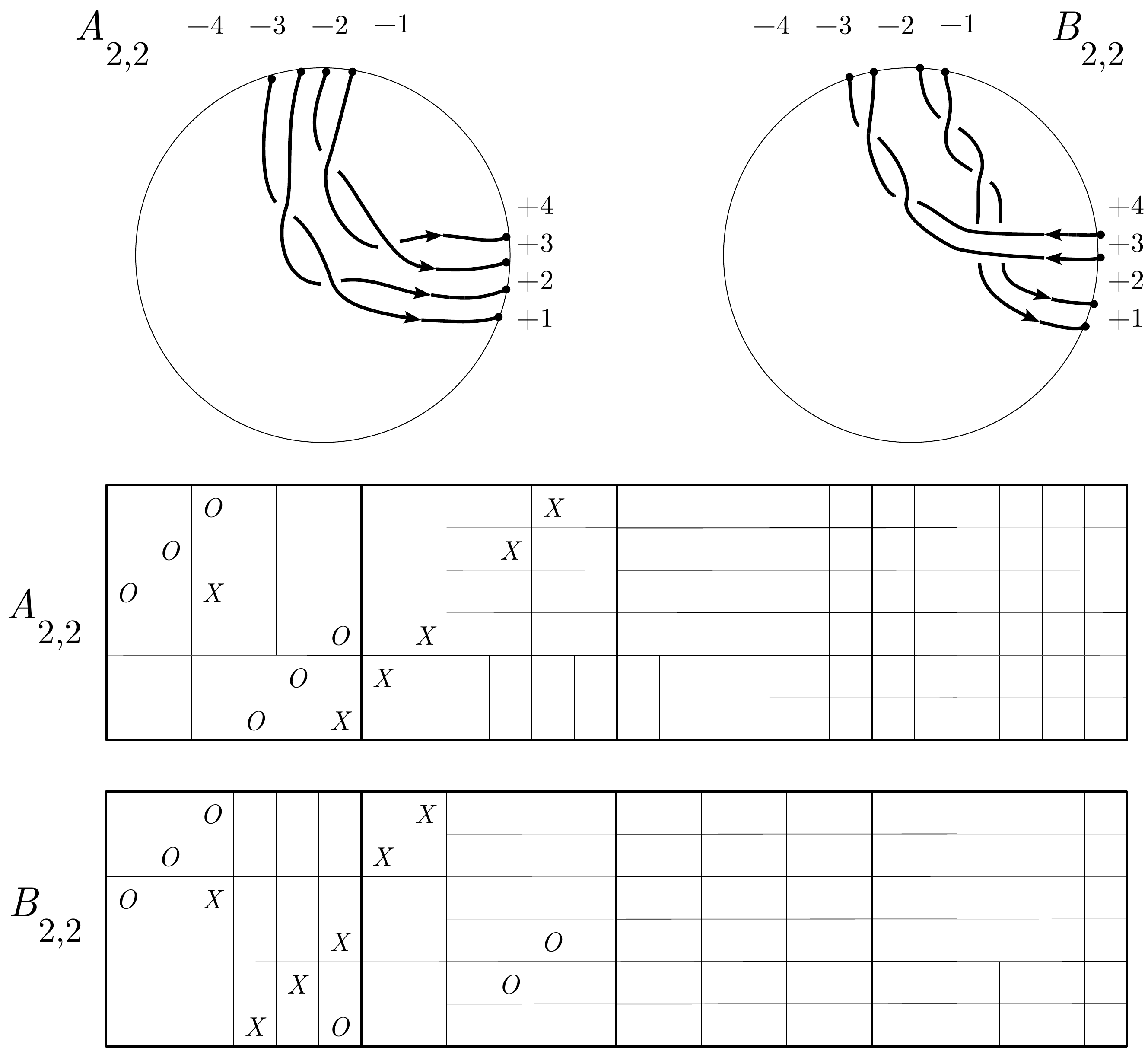}
\caption[legenda elenco figure]{Grid diagrams for different links in $L(4,1)$ with Hopf link lift.}\label{CE3grid}
\end{center}
\end{figure}

The lift of $U_{0,1,0,1}$ is the Hopf link, the lift of $\ U_{1,0,0,0}$ is the trivial link with four
component and  $ U_{0,2,1,0},\ U_{0,2,0,2},\ U_{0,0,1,2},\ U_{0,0,2,0}$ lif to   the closure of the braid  $\Delta_4^{2}$, where $\Delta_4$ denotes the Garside
braid on 4-strands (see \cite{M}). So, if we assume $J_{4,1}(L):=J_{1,0}(\widetilde{L})$ on trivial links, we get the following different  HOMFLY-PT polynomials
\begin{eqnarray*}
J_{4,1} (A_{2,2})&=&
 a^9 z^{-3} - 3 a^{11}z^{-3} + 
 3 a^{13} z^{-3} - a^{15} z^{-3} + 3 a^{25} z^{-2} - 9 a^{27} z^{-2} + \\ 
 &\ &+ 9 a^{29} z^{-2} - 3 a^{31} z^{-2} + 3 a^9 z^{-1} - 15 a^{11} z^{-1} + 
 21 a^{13} z^{-1} +
\\
& \ & - 9 a^{15} z^{-1} +
4 a^{25} - 12 a^{27}  +12 a^{29} - 4 a^{31}  
 + a^9 z - 25 a^{11} z + \\
 &\ & + 62 a^{13} z -  38 a^{15} z + 3 a^{25} z - 3 a^{27} z + a^{25} z^2 - 3 a^{27} z^2 + \\
 &\ &+ 3 a^{29} z^2 - a^{31} z^2 - 19 a^{11} z^3 + 102 a^{13} z^3 - 99 a^{15} z^3 + 
 7 a^{25} z^3+ \\
 &\ &- 4 a^{27} z^3 - 7 a^{11} z^5 +94 a^{13} z^5 - 155 a^{15} z^5 + 
 5 a^{25} z^5 - a^{27} z^5 +\\
 &\ & - a^{11} z^7 + 46 a^{13} z^7 - 129 a^{15} z^7 + 
 a^{25} z^7 +11 a^{13} z^9 - 56 a^{15} z^9+\\
 &\ &+ a^{13} z^{11} - 12 a^{15} z^{11} - 
 a^{15} z^{13}\\
 \end{eqnarray*}
 \begin{eqnarray*}
 J_{4,1} (B_{2,2})& =& 
a^9 z^{-3} - 3 a^{11}z^{-3} + 3 a^{13}z^{-3} - a^{15}z^{-3} + 
 2 a^5 z^{-2} - 6 a^7 z^{-2}+\\ 
&\ & + 6 a^9 z^{-2} - 2 a^{11} z^{-2} + a^{25} z^{-2} - 
 3 a^{27} z^{-2} + 3 a^{29} z^{-2} - a^{31} z^{-2} +\\
 &\ &+ 3 a^9 z{-1} - 15 a^{11} z^{-1} + 
 21 a^{13} z^{-1} -9 a^{15} z^{-1} + 
 2 a^5 - 18 a^7 + 30 a^9+\\
 &\ & - 14 a^{11} + 2 a^{25} - 6 a^{27} + 6 a^{29} - 
 2 a^{31} + 
 a^9 z - 25 a^{11} z + \\
&\ & + 62 a^{13} z - 38 a^{15} z +
  a^{25} z - a^{27} z - 20 a^7 z^2 + 70 a^9 z^2 +\\
  &\ &- 50 a^{11} z^2 + 
 a^{25} z^2 - 3 a^{27} z^2 + 3 a^{29} z^2 - a^{31} z^2 - 19 a^{11} z^3 + \\
&\ & + 102 a^{13} z^3 - 99 a^{15} z^3 + 3 a^{25} z^3 - 2 a^{27} z^3 - 10 a^7 z^4 
+ 88 a^9 z^4 +\\
&\ & - 110 a^{11} z^4 - 7 a^{11} z^5 + 94 a^{13} z^5 - 
 155 a^{15} z^5 + 3 a^{25} z^5 - a^{27} z^5 + \\
 &\ & - 2 a^7 z^6  + 58 a^9 z^6 - 
 128 a^{11} z^6 - a^{11} z^7 + 46 a^{13} z^7 - 129 a^{15} z^7 +\\ 
 &\ & + a^{25} z^7 + 
 18 a^9 z^8 - 74 a^{11} z^8 
 +11 a^{13} z^9 - 56 a^{15} z^9 + 2 a^9 z^{10} +\\ &\ &- 
 20 a^{11} z^{10} + a^{13} z^{11} - 12 a^{15} z^{11} - 2 a^{11} z^{12} - a^{15} z^{13}
 \end{eqnarray*}
\end{ese}
\end{section}

\begin{section}{Link Floer Homology in lens spaces}
In this section we generalize to the case of links a combinatorial description of the  hat version $\widehat{HFK}$ of the  Link Floer Homology developed
in \cite{BGH}
for knots in lens spaces. Then we compute it on some examples and discuss whether this invariant is  essential. We start by recalling some definitions.

\paragraph{The complex $\mathbf{(C(G) , \partial)}$}
Consider a grid diagram  $G=(T^2, \boldsymbol{\alpha} ,\boldsymbol{\beta} ,\mathbb{O} , \mathbb{X})$ representing an oriented knot in $L(p,q)$ and denote with $n$ 
its grid number.  
Following \cite{BGH}, we associate to $G$  a chain complex $(C(G) , \partial)$.
Let $\textbf{x}$ be an unordered $n$-uple of intersection points belonging to $\boldsymbol{\alpha} \cap \boldsymbol{\beta}$ such that each intersection point 
belongs to different curves  of $\boldsymbol{\alpha}$ and $\boldsymbol{\beta}$. Denote by $Y$ the set of these elements and 
let $C(G)$ be the $\mathbb{Z}_2$-module generated by the set $Y$.
Given $\mathbf x\in Y$,  we call   \textit{components} of $\textbf{x}$ the  points  of $\textbf{x}$ and we denote by $x_i$ 
the only component of $\textbf{x}$ laying on the $\alpha_i$ circle. 
If  $S_n$ is  the symmetric group on $n$ letters, 
there is a one to one correspondence between elements of $Y$ and those of $S_n \times \mathbb{Z}_p ^n$ (see Figure \ref{lorenzo1}).
Indeed, an element $( \sigma , (a_0, \ldots , a_{n-1} ) ) \in S_n \times \mathbb{Z}_p ^n$ corresponds to the only $\textbf{x}$ such that
\begin{itemize}
\item  $x_i$ lays on $\alpha_i \cap \beta_{\sigma (i)}$, for $i=0,\ldots,n-1$;
\item  $x_i$ is the $a_i$-th intersection of $\alpha_i \cap \beta_{\sigma (i)}$, for $i=0,\ldots,n-1$.
\end{itemize}
We use the notation $[c_0, \ldots ,c_{n-1}]$ to denote the permutation $\begin{pmatrix} 0 & \ldots & n-1 \\ c_0 & \ldots & c_{n-1} \end{pmatrix}$.

\begin {figure}[htb]
\begin{center}
 \includegraphics[width=15cm]{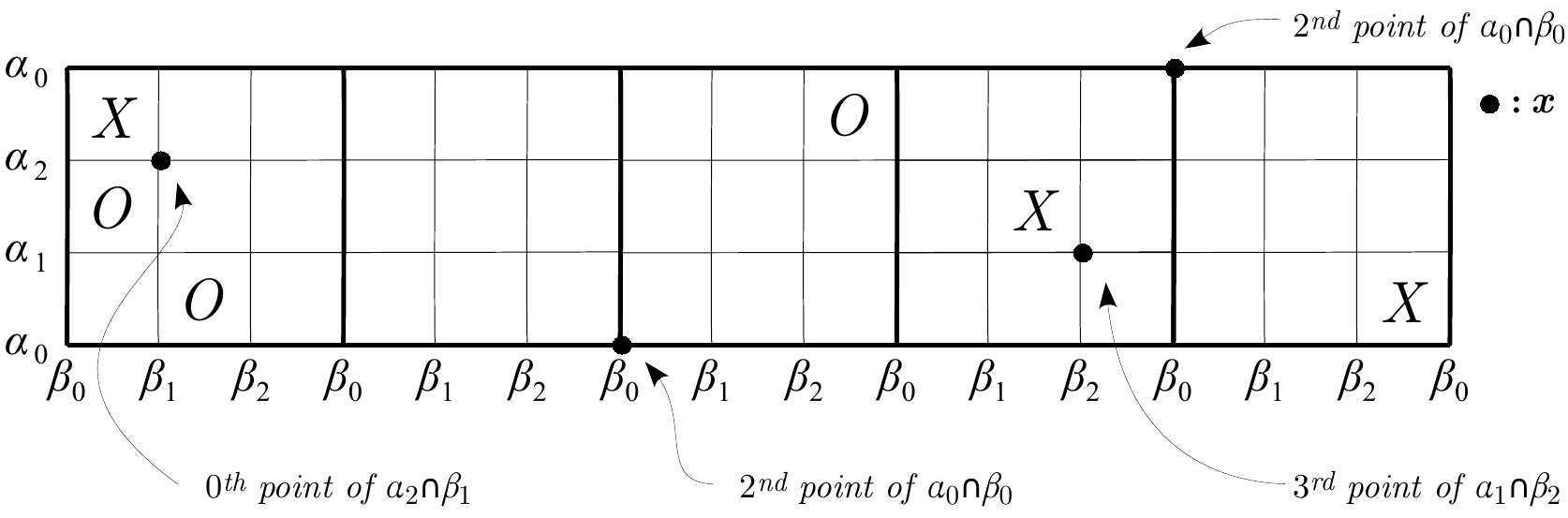}
 \end{center}
\caption {\textnormal{A generator $\mathbf x$ corresponding to $\{ [0 \ 2 \ 1] , (2,3,0) \} \in S_3 \times \mathbb{Z}_5 ^3$.}}\label{lorenzo1}
\end{figure}

Now we recall the definition of the  boundary operator. A \textit{parallelogram} is a quadrilateral properly embedded in $T^2$, that is a quadrilateral having
points of $\boldsymbol{\alpha} \cap \boldsymbol{\beta}$ as vertices and such that its sides coincide with arcs of curves belonging to $\boldsymbol{\alpha}$ or
$\boldsymbol{\beta}$.
Let $\textbf{x},\textbf{y} \in Y$ and let $P$ be a parallelogram; we say that a parallelogram $P$ \textit{connects} $\textbf{x}$ to $\textbf{y}$ if
\begin{itemize}
\item $\textbf{x}$ and $\textbf{y}$ differ for at  most two components $\{x_i,x_j\}$ and $\{y_i,y_j\}$ that are vertices of $P$;
\item  according to the orientation of $P$ induced by the one fixed on $T^2$, the sides of $P$ belonging to  $\boldsymbol{\alpha}$'s curves go from $\mathbf x$
vertices  to $\mathbf y$ ones.

\end{itemize}
We call $R(\textbf{x},\textbf{y})$ the set of parallelograms connecting $\textbf{x}$ to $\textbf{y}$.
We say that a parallelogram connecting $\textbf{x}$ to $\textbf{y}$ is \textit{admissibile} if its interior contains 
neither $\textbf{x}$ components nor $\textbf{y}$ ones.
For each pair of generators $\textbf{x},\textbf{y} \in Y$, we call $PG ( \textbf{x},\textbf{y})$ the set of admissible parallelograms 
connecting $\textbf{x}$ to $\textbf{y}$.
Given a parallelogram $P$, denote with  $n_{\mathbb{O}}(P)$ and $n_{\mathbb{X}}(P)$, respectively, the number of  $O$ markings and $X$ markings  belonging to $P$ 
(see Figure \ref{lorenzo2}).  
\begin {figure}[htb]
\begin{center}
 \includegraphics[width=14cm]{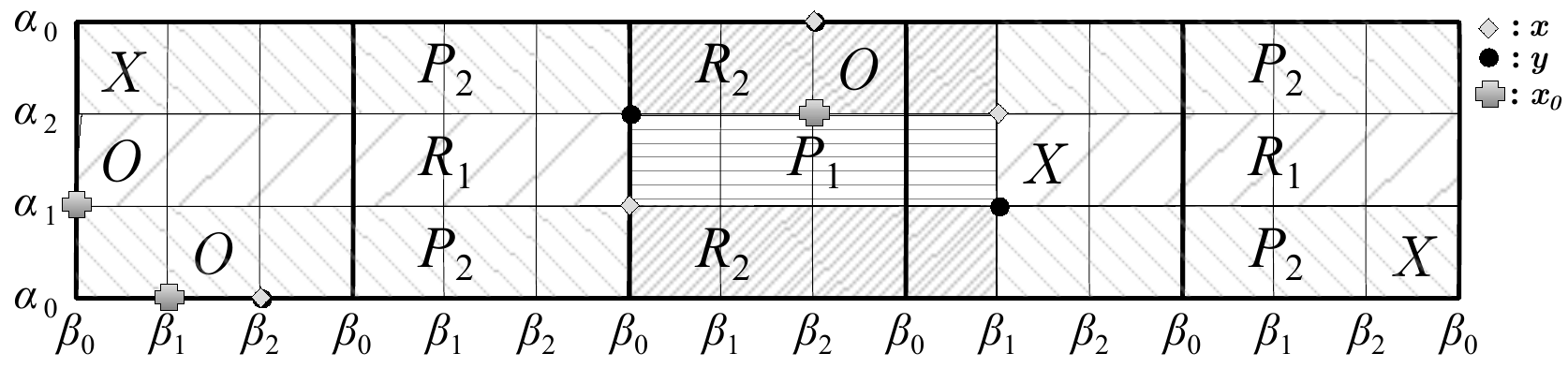}
\caption { \textnormal{Parallelograms $P_1$ and $P_2$ connect $\textbf{x}$ to $\textbf{y}$, while $R_1$ and $R_2$  connect $\textbf{y}$ to $\textbf{x}$.
Both $P_1$ and $R_1$ are admissible, while $P_2$ and $R_2$ are not. Moreover $n_{\mathbb{O}} (P_1) = n_{\mathbb{X}}(P_1)=0$, so the boundary operator
connects $\textbf{x}$ to $\textbf{y}$.}}\label{lorenzo2}
\end{center} 
\end{figure}\\
Now we are ready to define a boundary operator $\partial \colon C(G) \rightarrow C(G)$
\begin{eqnarray*}
\partial \textbf{x} = {\underset{\textbf{y} \in Y} {\sum}} \    {\underset{ \scriptsize{\left\{ \begin{array}{l} \ \ P \in PG (\textbf{x},\textbf{y}) : \\
  n_{\mathbb{O}} (P) = n_{\mathbb{X}} (P) =0 
  \end{array}\right\}}}{\sum} } \textbf{y}.
\end{eqnarray*} \\
Since $\partial^2=0$ (see \cite{BGH}), we can define the homology $H(C(G))$ associated to the chain complex $(C(G) , \partial)$,
obtaining  a bigraded $\mathbb{Z}_2$-vector space.

\paragraph{Defining degrees}
We can associate to each generator of $C(G)$ three different degrees: the spin degree, the Maslov degree and the Alexander degree.

Let $\textbf{x}_{\mathbb{O}} \in Y$ be the generator whose components are the  lower left vertices of the $n$ distinct parallelograms in $T^2 -\boldsymbol{\alpha
} - \boldsymbol{\beta}$ which contain elements of $\mathbb{O}$. Let $(\sigma_{\mathbb{O}} , (a_0, \ldots , a_{n-1} ) )$ be the element 
of $S_n \times \mathbb{Z}_p ^n$ corresponding to $\textbf{x}_{\mathbb{O}}$ and let $(\sigma, (b_0, \ldots , b_{n-1} ))$ be the element corresponding 
to a generic $\textbf{x}$. 
The \textit{spin degree}  is given by the function $\textbf{S}\colon Y \rightarrow \mathbb{Z}_p$ defined by
\begin{eqnarray}
\label{foS}
\textbf{S}(\textbf{x}) \equiv [q-1 + (    \underset{i=0}{ \overset{n-1} {\sum}} b_i -  \underset{i=0}{ \overset{n-1} {\sum}} a_i      ) ] \mod p.
\end{eqnarray}
The  \textit{Maslov degree} is the  function $\textbf{M}\colon Y\rightarrow \mathbb{Q}$ defined by 
\begin{equation}\label{foM}
\begin{array}{lll}
\textbf{M}(\textbf{x})&= &\dfrac{1}{p} ( I(\widetilde{W} (\textbf{x} ) ,\widetilde{W} (\textbf{x} )) - I(\widetilde{W} (\textbf{x} ) ,\widetilde{W} (\mathbb{O} )) -
I(\widetilde{W} (\mathbb{O} ) ,\widetilde{W} (\textbf{x} )) + \\
&\ &+ I(\widetilde{W} (\mathbb{O} ) ,\widetilde{W} (\mathbb{O} )) + 1 ) + d(p,q,q-1) + \dfrac{p-1}{p},
\end{array}
\end{equation}
where $d$ is a kind of normalization function depending only on lens space parameters, while $I$  and $\widetilde{W}$ are two functions  depending on the
arrangement of the $\mathbf x$ points with respect to the $\mathbb O$
points (for details see the Appendix).\\

Finally, the \textit{Alexander degree} is a  function $\textbf{A}\colon Y\rightarrow \mathbb{Q}$ defined by 
\begin{align}\label{foA}
\textbf{A}(\textbf{x}) = \dfrac{1}{2} ( \textbf{M}_{\mathbb{O} } ( \textbf{x}) - \textbf{M}_{\mathbb{X} } ( \textbf{x})- (n-1)),
\end{align}
where $\textbf{M}_{\mathbb{O} }$ is the Maslov degree and  $\textbf{M}_{\mathbb{X} }$ is the degree obtained by replacing $\mathbb O$ with $\mathbb X$ in 
formula (\ref{foM}). 

The  following equations show how the boundary operator relate with these degrees
\begin{eqnarray}
\label{fod}
\mathbf{S} ( \partial (\mathbf{x})) = \mathbf{S} (\mathbf{x})\qquad
\mathbf{M} ( \partial (\mathbf{x})) = \mathbf{M} (\mathbf{x})-1 \qquad
\mathbf{A} (\partial (\mathbf{x})) = \mathbf{A} (\mathbf{x}).
\end{eqnarray}

\paragraph{Knot Floer Homology}
Let $V$ be a bidimensional $\mathbb{Z}_2$-vector space spanned by a vector  with Maslov-Alexander bigrading $(-1,-1)$ and another one
 with Maslov-Alexander bigrading $(0,0)$.
\begin{prop}[\cite{BGH, MOS}]\label{HFK}
Consider a  grid diagram $G_K$ of an oriented knot $K \subset L(p,q)$. Then $H(C(G_K) , \partial)$ 
is isomorphic to the bigraded group $\widehat{HFK} (K) \otimes V^{\otimes (n-1)}$, where $n$ is the grid number of $G$. 
\end{prop}

\paragraph{Link Floer Homology}

In this paragraph we generalize the combinatorial computation of $\widehat{HFL}$ to the case of links. 
Let $L \subset L(p,q)$ be an oriented link with $l$ components $L_1,\ldots,L_l$ and let
$G_L=(T^2, \boldsymbol{\alpha} ,\boldsymbol{\beta} ,\mathbb{O} , \mathbb{X})$ be a grid diagram of it. 
Let $k_j$ be the number of $\mathbb O$ markings belonging to $L_j$ (which is equal to the one of $\mathbb X$ markings) and denote 
elements of  $\mathbb{O}$ or $\mathbb{X}$ 
belonging to $L_j$ with as   $O_{j,i}$ or  $X_{j,i}$ for  $j=1, \ldots ,l$
and  $i=1, \ldots ,k_j$.
The generators of $C(G_L)$, the boundary operator, the spin degree and the Maslov degree are defined as in the case of knots. Instead, the Alexander degree becomes 
a multidegree
 as follows. Consider the set $\mathbb{O}_j$ composed by  elements of $\mathbb{O}$ belonging  to  $L_j$. 
 Let $\textbf{M}_{\mathbb{O}_j } ( \textbf{x})$ be the Maslov degree of $\textbf{x}$, computed with respect to $\mathbb{O}_j$.
The Alexander multidegree is the function $\textbf{A}\colon Y \rightarrow {\mathbb{Q}}^l$ defined by
\begin{align*}
\textbf{A}(\textbf{x}) = \dfrac{1}{2} \Big( \textbf{M}_{\mathbb{O}_1 } ( \textbf{x}) - \textbf{M}_{\mathbb{X}_1 } ( \textbf{x})- (n_1 -1), \ldots ,  \textbf{M}_{\mathbb{O}_l } \big( \textbf{x}) - \textbf{M}_{\mathbb{X}_l } ( \textbf{x})- (n_l -1)\Big).
\end{align*}
As in the case of knots, we can define the homology $H(C(G_L) , \partial)$ of the chain complex $(C(G_L) , \partial)$.
For $j=1,\ldots , l$, let $V_j$ be a bidimensional $\mathbb{Z}_2$-vector space, spanned by  a vector  with Maslov-Alexander multidegree 
$(0,(0,0, \ldots , 0) )$ and  another one with multidegree $(-1,- \vec{e_j})$, where $\vec{e_j}$ indicates the $j$-th vector of  the  canonical
basis of $\mathbb{R}^l$.
\begin{prop}
\label{proplorenzo1}
Let  $L$ be an oriented link in $L(p,q)$, and let $G_L$ be a  grid diagram of $L$. Denote with $L_1,\ldots,L_l$ the components of $L$. Then
\begin{eqnarray*}
 H(C(G_L) , \partial) \cong \widehat{HFL} (L) \otimes \underset{j=1}{ \overset{l} {\bigotimes}} V_j^{\otimes (k_j -1)}, 
\end{eqnarray*}
where $k_j$ is the number of  $O$ markings belonging to $L_j$.
\end{prop}
\proof
Since Proposition 7.2 of \cite{OS2} holds also in the case of lens spaces, by using an argument similar to the one used in proof of Proposition 2.5 of \cite{MOS}, we can conclude.
\endproof

\paragraph{Behavior under change of orientation}
Let $G_K$ be a grid diagram  of an oriented  knot $K\subset L(p,q)$ and let $-G_K$ be a grid diagram of $-K$, 
obtained  exchanging the elements of $\mathbb{O}$ and $\mathbb{X}$ in $G_K$.
\begin{prop}
\label{HFLorient}
There is a one to one correspondence between the generators of $H(C(G_K),\partial)$ and those of $H(C(-G_K),\partial)$:
a generator of  $H(C(G_K),\partial)$ having spin degree $s$, Maslov degree $m$ and Alexander degree $a$ corresponds to a generator
of  $H(C(-G_K),\partial)$ with spin degree $s+k$, Maslov degree $m-2a-(n-1)$ and Alexander degree $-a-(n-1)$, where $k$ is a fixed 
integer and $n$ denote the grid number of $G_K$.
\end{prop}
\proof
Clearly the generators of $C(G_K)$ coincide with  those of $C(-G_K)$, but for simplicity's sake, given $\mathbf x\in C(G_K)$
we  denote with $-\mathbf x$ the same generator thought in
$C(-G_K)$. Moreover,  two generators $\mathbf{x},\mathbf{y} \in (C(G_K),\partial)$ are connected by 
the boundary operator if and only if $-\mathbf{x}$ and $-\mathbf{y}$ are connected by the boundary operator in the chain complex $(C(-G_K),\partial)$. 
Thus generators of $H(C(G_K),\partial)$ coincide with those of $H(C(-G_K),\partial)$. On the contrary,
$\mathbf x$ and $-\mathbf x$  have generally different degrees.
By definition we have  $$\mathbf M(\mathbf -x)=\mathbf{M_{\mathbb{O}}}(\mathbf{-x})=\mathbf{M_{\mathbb{X}}}(\mathbf{x})=
\mathbf M_{\mathbb O}(\mathbf x)-2A(\mathbf x)-(n-1)$$ and 

\begin{align*}
\mathbf{A}(-\mathbf{x})=\dfrac{1}{2}( \mathbf{M_{\mathbb{O}}}(-\mathbf{x}) - \mathbf{M_{\mathbb{X}}}(-\mathbf{x}) -(n-1) )=\dfrac{1}{2}( \mathbf{M_{\mathbb{X}}}(\mathbf{x}) - \mathbf{M_{\mathbb{O}}}(\mathbf{x}) -(n-1) )= \\
= \dfrac{1}{2}(      \mathbf{M_{\mathbb{O}}}(\mathbf{x}) -2 \mathbf{A} (\mathbf{x}) -(n-1) - \mathbf{M_{\mathbb{O}}}(\mathbf{x}) -(n-1) )= -\mathbf{A} (\mathbf{x}) -(n-1).
\end{align*}
 
Let  $\textbf{x}_{\mathbb{O}}=\{ \sigma_{\mathbb{O}} , (b_0, \ldots , b_{n-1} ) \}$ (resp. $\textbf{x}_{\mathbb{X}}=\{ \sigma_{\mathbb{X}} ,
(c_0, \ldots , c_{n-1} ) \}$) 
be the generators of $C(G_K)$ whose components are lower left vertices of the $n$ distinct parallelograms in $T^2-\mathbf{\alpha}-\mathbf{\beta}$ containing elements 
of $\mathbb O$ (resp. $\mathbb X$). Set $k:=\underset{i=0}{ \overset{n-1} {\sum}} b_i -  \underset{i=0}{ \overset{n-1} {\sum}} c_i   \mod p$. Then we have
 \begin{align*}
 \mathbf{S} (-\mathbf{x})= \mathbf{S} (\mathbf{x}) + k.
 \end{align*}
\endproof
Observe that, according to Proposition \ref{proplorenzo1}, if $G_K$ has grid number $1$, then $\widehat{HFK} (K) = H(C(G_K),\partial)$, 
as a consequence $\widehat{HFK} (-K)$ can be achieved straightly from $\widehat{HFK} (K)$.
\begin{ese}
\label{eseorient}
 In the Appendix, we compute the Knot Floer Homology of both the knots $K_1$ and $-K_1$ depicted in   Figure \ref{CE1grid}, using a grid diagram with grid number $1$.  We obtain
\begin{align*}
& \widehat{HFK} (K_1) \cong \mathbb{Z}_2 [0,-\dfrac{2}{5} , -\dfrac{1}{5}] \oplus \mathbb{Z}_2 [1,-\dfrac{2}{5} , -\dfrac{2}{5}] \oplus \mathbb{Z}_2 [2,\dfrac{2}{5} , \dfrac{2}{5}] \oplus \mathbb{Z}_2 [3,0 , \dfrac{1}{5}] \oplus \\
& \oplus \mathbb{Z}_2 [4,\dfrac{2}{5} , 0]
\end{align*}
and 
\begin{align*}
& \widehat{HFK} (-K_1) \cong \mathbb{Z}_2 [2,0,\dfrac{1}{5}]
\mathbb{Z}_2 [3 ,\dfrac{2}{5}, \dfrac{2}{5}] \oplus 
\mathbb{Z}_2 [4,-\dfrac{2}{5} , -\dfrac{2}{5}]\oplus
 \mathbb{Z}_2 [0,-\dfrac{2}{5} , -\dfrac{1}{5}]\\&\  \oplus \mathbb{Z}_2 [1,\dfrac{2}{5} , 0]
\end{align*}
where $\mathbb{Z}_2 [i, j, k]$ 
denotes a $\mathbb{Z}_2$-vector space spanned by a generator  with spin degree $i$, Maslov degree $j$ and Alexander degree $k$. This
value of $\widehat{HFK} (-K_1) $ clearly coincides with the one obtained using Proposition \ref{HFLorient} (with $k=2$).
\end{ese}

To end this section we compute $\widehat{HFL}$ on the pairs of links of Examples \ref{perlorenzo} and \ref{perlorenzo2} in order to test whereas 
 this invariant is essential or not. Computations are very long so we report them in
the Appendix, while here we collect only the results, showing that $\widehat{HFL}$  can distinguish both the pairs of links.

\begin{ese}
\label{esHFL}
Let  $K_1$ and $K_2$ be the two non equivalent  knots in $L(5,2)$, depicted in Figure \ref{CE1grid},   both lifting to the trivial knots in $\s3$ (see \cite{M}). 
We have 

\begin{align*}
& \widehat{HFK} (K_1) \cong \mathbb{Z}_2 [0,-\dfrac{2}{5} , -\dfrac{1}{5}] \oplus \mathbb{Z}_2 [1,-\dfrac{2}{5} , -\dfrac{2}{5}] \oplus \mathbb{Z}_2 [2,\dfrac{2}{5} , \dfrac{2}{5}] \oplus \mathbb{Z}_2 [3,0 , \dfrac{1}{5}] \oplus \\
& \oplus \mathbb{Z}_2 [4,\dfrac{2}{5} , 0]
\end{align*}
and
\begin{align*}
&\widehat{HFK} (K_2)  \cong \mathbb{Z}_2 [0,-\dfrac{2}{5} , 0] \oplus \mathbb{Z}_2 [1,-\dfrac{2}{5} , -\dfrac{2}{5}] \oplus 
\mathbb{Z}_2 [2,\dfrac{2}{5} , \dfrac{1}{5}] \oplus \mathbb{Z}_2 [3,0 , -\dfrac{1}{5}] \oplus \\
& \oplus \mathbb{Z}_2 [4,\dfrac{2}{5} , \dfrac{2}{5}].
\end{align*}

\end{ese}

\begin{ese}
\label{esHFL2}
Let $L_A$ and $L_B$ be the two non equivalent links in $L(4,1)$, depicted in  Figure \ref{CE2grid}  both lifting  to the Hopf link  in $\s3$ (see \cite{M}). 
We have
\begin{align*}
& \widehat{HFL} (L_A) \cong  \mathbb{Z}_2 [0,\dfrac{1}{2},\dfrac{1}{2}] \oplus \mathbb{Z}_2 [1,\dfrac{1}{2},-\dfrac{1}{2}] \oplus \mathbb{Z}_2 [2,-\dfrac{1}{2},\dfrac{1}{2}] \oplus \\ 
& \oplus \mathbb{Z}_2 [3,-\dfrac{1}{2},-\dfrac{1}{2}]
\end{align*}
and
\begin{align*}
& \widehat{HFL} (L_B)  \cong H(C(G_B) , \partial) \cong \mathbb{Z}_2 \left[0,\dfrac{1}{4} , (\dfrac{1}{8},\dfrac{1}{8} ) \right] \oplus\mathbb{Z}_2 \left[0,-\dfrac{3}{4} , (-\dfrac{7}{8},\dfrac{1}{8} ) \right] \oplus \\
& \oplus \mathbb{Z}_2 \left[0,-\dfrac{3}{4} , (\dfrac{1}{8},-\dfrac{7}{8} ) \right] \oplus \mathbb{Z}_2 \left[0, -\dfrac{7}{4} , (-\dfrac{7}{8},-\dfrac{7}{8}) \right] \oplus\mathbb{Z}_2 \left[1, 0 , (-\dfrac{5}{8},-\dfrac{1}{8}) \right] \oplus \\ 
& \oplus  \mathbb{Z}_2 \left[1, -1 , (-\dfrac{5}{8},-\dfrac{1}{8}) \right] \oplus \mathbb{Z}_2 \left[2, \dfrac{1}{4} , (-\dfrac{3}{8},-\dfrac{3}{8}) \right] \oplus {\mathbb{Z}_2 \left[2, -\dfrac{3}{4} , (-\dfrac{3}{8},-\dfrac{3}{8}) \right]}^3 \oplus \\ 
& \oplus \mathbb{Z}_2 \left[3, 0 , (-\dfrac{1}{8},-\dfrac{5}{8}) \right] \oplus \mathbb{Z}_2 \left[3,-1 , (-\dfrac{1}{8},-\dfrac{5}{8}) \right]
\end{align*}
where  $\mathbb{Z}_2 [i, j , (k_1,k_2)]$ denotes a $\mathbb{Z}_2$-vector spanned by a generator   with spin degree $i$, Maslov degree $j$ and Alexander bigrading $(k_1,k_2)$.\\

\end{ese}

\end{section}

\begin{section}{Appendix}

This appendix contains the computations of Examples \ref{esHFL} and \ref{esHFL2}.

\paragraph{Maslov index}
First of all we recall the definition of the functions $d,I$ and $\widetilde{W}$ appearing in the formula (\ref{foM}) of the Maslov index 
(see \cite{BGH} for details).
In order to compute these functions it is more easy to keep slanted grid diagrams. 

Let $G=(T^2, \boldsymbol{\alpha} ,\boldsymbol{\beta} ,\mathbb{O} , \mathbb{X})$ be 
a grid diagram representing a link in $L(p,q)$ and let $n$ be its grid number. We denote with $d\colon\mathbb{Z} \times \mathbb{Z} \times \mathbb{Z} \rightarrow \mathbb{Q} $ be the function defined by induction as
\begin{align*}
d(1,0,0)=0  \qquad \qquad \qquad \qquad \qquad \qquad \ \quad \qquad \\
d(p,q,i)= \left( \dfrac{pq - {(2i +1 -p-q)}^2}{4pq} \right) - d(q,r,j)
\end{align*}
where $r \equiv p \mod q$ and $j \equiv i \mod q$.
Consider the function
\begin{align*}
   W\colon \left\{ \begin{array}{l}
        \textrm{Finite set of} \\
         \textrm{points in $G$}
   \end{array}\right\} \rightarrow 
       \left\{ \begin{array}{l}
       \textrm{Finite set of pairs $(a,b)$} \\
        \textrm{with $a\in [0,pn) , b\in [0,n)$}
    \end{array}\right\} 
\end{align*}
that associates to a $n$-ple of  points of  $G$  their coordinates in $\mathbb{R}^2$ with respect to the base
\begin{align*}
 \left( \begin{array}{l}
      \vec{v}_1 = \left(\dfrac{1}{np}, 0 \right) ,  \vec{v}_2 = \left(- \dfrac{q}{np}, \dfrac{1}{n} \right)    \end{array}\right).
\end{align*}
Assume that the points of $\mathbb{O}$ and $\mathbb{X}$ are placed in the centre of their respective parallelograms. 
In this way, with respect to the basis $(\vec{v}_1, \vec{v}_2)$,the generators $\textbf{x}$ have integer coordinates, 
whereas the  points of $\mathbb{O}$ and $\mathbb{X}$ have rational coordinates.
Now define the function
\begin{align*}
    C_{p,q}\colon \left\{ \begin{array}{l}
        \textrm{Finite sets of pairs (a,b)} \\
         \textrm{where} \ a \in [0,pn), b \in [0,n)  
    \end{array}\right\} \rightarrow 
        \left\{ \begin{array}{l}
        \textrm{Finite sets of pairs $(a,b)$} \\
            \textrm{where \ $a,b \in [0,pn)$}
   \end{array}\right\} 
\end{align*}
that, to a $n$-uple of coordinates 
\begin{eqnarray*}
\left( (a_i , b_i ) \right)_{i=0} ^{n-1}
\end{eqnarray*}
associates a $pn$-uple of coordinates
\begin{eqnarray*}
(a_i +nqk \mod np, b_i + nk)_{i=0,k=0} ^{i=n-1,k=p-1} .
\end{eqnarray*}
Let $A$ and $B$ be two finite sets of pairs of coordinates and  $I$ be the function that, 
to a ordinate pair $(A,B)$, associates the cardinality of the set of the pairs $(a,b) \in A\times B$, $a=(a_1,a_2) \in A , \ b=(b_1,b_2) \in B$, 
such that $a_i <b_i$ for $i=1,2$. Define  $\widetilde{W}:= C_{p,q} \circ W$.

\paragraph{Computation of examples \ref{eseorient} and \ref{esHFL} }
To compute the Link Floer Homology of the oriented knots $K_1,K_2\subset L(5,2)$ depicted in Figure  \ref{CE1grid}, 
we use the slanted grid diagrams $G_1$ and $G_2$ depicted in Figure 19.
\begin {figure}[ht]
\label{fin1}
\begin{center}
\includegraphics[width=5in]{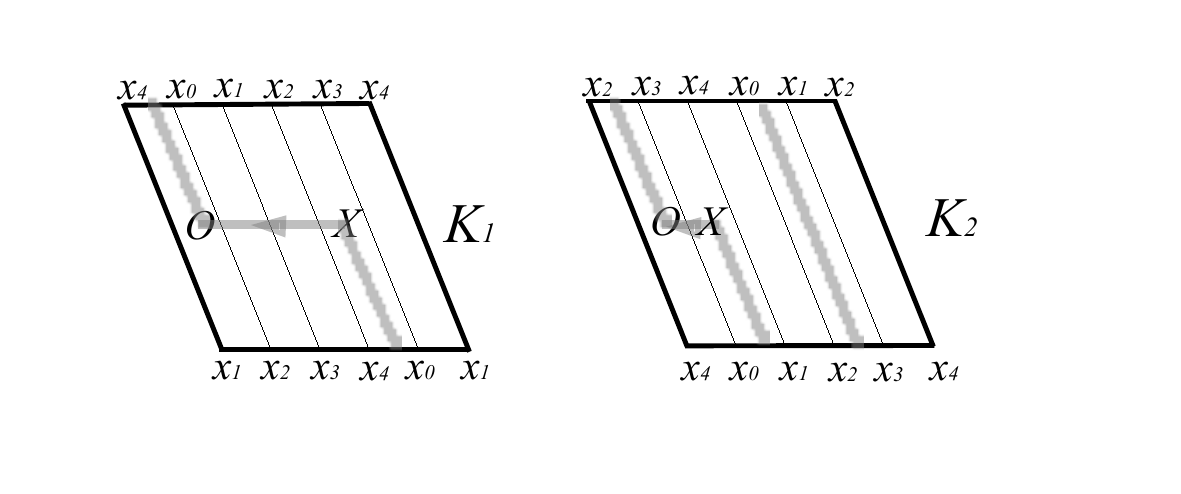}
\caption {\textnormal{Grid diagrams $G_1$ and $G_2$.}}
\end{center}
\end{figure}
Both the of generators  of both $C(G_1)$ and  $C(G_2)$ are in one to one correspondence with $S_1 \times \mathbb{Z}_5$ 
and hence they consist of five elements that we denote with $\big\{ \{[0] ,(0)\}, \{[0] ,(1)\},\{[0] ,(2)\},\{[0] ,(3)\},\{[0] ,(4)\} \big\}$.

Let us compute the spin degree of each generator.
First, observe that for both $G_1$ and $G_2$, we have that $\{[0] ,(0)\}=x_{\mathbb{O}}$. So, from formula (\ref{foS}), we get
 \begin{eqnarray*}
 \mathbf{S}(\{[0] ,(i)\}) \equiv [ 2-1 + (i-0)] \equiv 0 \mod 5, \ \forall i=0,\ldots, 4.
 \end{eqnarray*}
Then, if  $x_i$ denotes the generator having spin degree $i$, we have
 \begin{eqnarray*}
 x_0 := \{[0] ,(4)\}, \ x_1 := \{[0] ,(0)\}, \ x_2 := \{[0] ,(1)\}, \ x_3 := \{[0] ,(2)\}, \ x_4 := \{[0] ,(3)\}.
 \end{eqnarray*}

Now we deal with  the Maslov degree  starting by computing the values assumed by the function $d$
\begin{align*}
d(p,q,q-1)=d(5,2,1)=\left( \dfrac{10- {(2+1-5-2)}^2}{4 \cdot 5 \cdot 2} \right) - d(2,1,1) = \\
=-\dfrac{6}{40} - d(2,1,1) = -\dfrac{3}{20} - \left( \dfrac{2- {(2+1-2-1)}^2}{4 \cdot 2 \cdot 1} \right) - d(1,0,0)= \\
=-\dfrac{3}{20}-\dfrac{1}{4}-0=-\dfrac{2}{5}. \quad \qquad \qquad \qquad \qquad \qquad \qquad \qquad \qquad \qquad 
 \end{align*}

Denote with $Z$ the $X$ marking of $G_2$. The elements $x_i , O, X$ and $Z$, considered as points of the unitary square $[0,1]\times[0,1]$, have the following coordinates with respect to the canonical basis
\vspace{-0.2in}
\begin{center}
\begin{tabular}{ccccc}
$x_0= \left( \dfrac{4}{5},0 \right)$ & $x_1= \left( 0,0 \right)$ & $x_2= \left( \dfrac{1}{5},0 \right)$ & $x_3= \left( \dfrac{2}{5},0 \right)$ & $x_4= \left( \dfrac{3}{5},0 \right)$\\
\rule[-4mm]{0mm}{1.2cm}
$O=\left( \dfrac{9}{10}, \dfrac{1}{2} \right)$ & $X=\left( \dfrac{1}{2}, \dfrac{1}{2} \right)$ & $Z=\left( \dfrac{1}{10}, \dfrac{1}{2} \right)$. & &
\end{tabular}
\end{center}
The grid diagrams $G_1$ and $G_2$ have grid number $n=1$, so, switching to the basis 
 $\left( \vec{v}_1 = \left(\dfrac{1}{np}, 0 \right)=\left(\dfrac{1}{5}, 0 \right) ,  \vec{v}_2 =  \left(-\dfrac{q}{np}, \dfrac{1}{n} \right)=\left(- \dfrac{2}{5}, 1 \right) \right)$
 by means of the basis change matrix  $W= \left( \begin{matrix}  np&nq\\ 0&n \end{matrix} \right)=\left( \begin{matrix}  5&2\\ 0&1 \end{matrix} \right)$, we get the 
 following values.

\begin{center}
\begin{tabular}{|c|c|c|c|c|c|c|c|c|} 
\hline
& $x_0$ & $x_1$ & $x_2$ & $x_3$ & $x_4$ & $O$ & $X$ & $Z$ \\
\hline
\rule[-4mm]{0mm}{1cm}
W(column) &  $\left( 4,0 \right)$ & $\left( 0,0 \right)$ & $\left( 1,0 \right)$ & $\left( 2,0 \right)$ & $\left( 3,0 \right)$ &
$\left( \dfrac{11}{2}, \dfrac{1}{2} \right)$ & $\left( \dfrac{7}{2}, \dfrac{1}{2} \right)$ & $\left( \dfrac{3}{2}, \dfrac{1}{2} \right)$\\
\hline
\end{tabular}
\end{center}
By composing with the function $C_{5,2}$ we have
\vspace{-0.7cm}
\begin{footnotesize}
\begin{center}
\begin{tabular}{c}
$\widetilde{W}(x_0)$ = $( \left( 4 \negthickspace \negthickspace \mod 5,0 \right), \left( 1 \negthickspace \negthickspace \mod 5,1 \right), \left( 3 \negthickspace \negthickspace \mod 5,2 \right),\left( 0 \negthickspace \negthickspace \mod 5,3 \right),\left( 2 \negthickspace \negthickspace \mod 5,4 \right) )$ \\
$\widetilde{W}(x_1)$ =  $( \left( 0  \negthickspace \negthickspace \mod 5,0 \right), \left( 2  \negthickspace \negthickspace  \mod 5,1 \right), \left( 4  \negthickspace \negthickspace  \mod 5,2 \right),\left( 1  \negthickspace \negthickspace  \mod 5,3 \right),\left( 3  \negthickspace \negthickspace  \mod 5,4 \right) )$ \\
$\widetilde{W}(x_2)$ =  $( \left( 1  \negthickspace \negthickspace  \mod 5,0 \right), \left( 3  \negthickspace \negthickspace  \mod 5,1 \right), \left( 0  \negthickspace \negthickspace  \mod 5,2 \right),\left( 2  \negthickspace \negthickspace  \mod 5,3 \right),\left( 4  \negthickspace \negthickspace  \mod 5,4 \right) )$ \\
$\widetilde{W}(x_3)$ =  $( \left( 2  \negthickspace \negthickspace  \mod 5,0 \right), \left( 4  \negthickspace \negthickspace  \mod 5,1 \right), \left( 1  \negthickspace \negthickspace  \mod 5,2 \right),\left( 3  \negthickspace \negthickspace  \mod 5,3 \right),\left( 0  \negthickspace \negthickspace  \mod 5,4 \right) )$\\
$\widetilde{W}(x_4)$ =  $( \left( 3  \negthickspace \negthickspace  \mod 5,0 \right), \left( 0  \negthickspace \negthickspace  \mod 5,1 \right), \left( 2  \negthickspace \negthickspace  \mod 5,2 \right),\left( 4  \negthickspace \negthickspace  \mod 5,3 \right),\left( 1  \negthickspace \negthickspace  \mod 5,4 \right) )$ \\
\rule[-4mm]{0mm}{1cm}
$\widetilde{W}(O)$ = $\Big( \left( \dfrac{1}{2}  \negthickspace \negthickspace  \mod 5,\dfrac{1}{2} \right), \left( \dfrac{5}{2}  \negthickspace \negthickspace  \mod 5,\dfrac{3}{2} \right), \left( \dfrac{9}{2}  \negthickspace \negthickspace  \mod 5,\dfrac{5}{2} \right), \left( \dfrac{3}{2}  \negthickspace \negthickspace  \mod 5,\dfrac{7}{2} \right),\left( \dfrac{7}{2} \negthickspace \negthickspace  \mod 5,\dfrac{9}{2} \right) \Big)$    \\
\rule[-4mm]{0mm}{1cm}
$\widetilde{W}(X)$ = $\Big( \left( \dfrac{7}{2}  \negthickspace \negthickspace \mod 5,\dfrac{1}{2} \right), \left( \dfrac{1}{2}  \negthickspace \negthickspace \mod 5,\dfrac{3}{2} \right), \left( \dfrac{5}{2}  \negthickspace \negthickspace \mod 5,\dfrac{5}{2} \right), \left( \dfrac{9}{2}  \negthickspace \negthickspace \mod 5,\dfrac{7}{2} \right),\left( \dfrac{3}{2}  \negthickspace \negthickspace \mod 5,\dfrac{9}{2} \right) \Big)$ \\
\rule[-4mm]{0mm}{1cm}
$\widetilde{W}(Z)$ =  $\Big( \left( \dfrac{3}{2}  \negthickspace \negthickspace \mod 5,\dfrac{1}{2} \right), \left( \dfrac{7}{2}  \negthickspace \negthickspace \mod 5,\dfrac{3}{2} \right), \left( \dfrac{1}{2}  \negthickspace \negthickspace \mod 5,\dfrac{5}{2} \right), \left( \dfrac{5}{2}  \negthickspace \negthickspace \mod 5,\dfrac{7}{2} \right),\left( \dfrac{9}{2}  \negthickspace \negthickspace \mod 5,\dfrac{9}{2} \right) \Big)$\\
\end{tabular}
\end{center}
\end{footnotesize}
and so we obtain

\fontfamily{phv} \selectfont
\begin{center}
\begin{tabular}{|c|c|c|c|c|c|c|c|c|}
\hline 
 $J(row,column)$ & $x_0$ & $x_1$ & $x_2$ & $x_3$ & $x_4$ & $O$ & $X$ & $Z$ \\
\hline
 $x_0$ &  $3$ &   &   &   &   & $10$ & $7$ & $8$ \\
\hline
$x_1$ &  & $7$  &   &   &  & $12$ & $8$ & $8$ \\
\hline
$x_2$ &  &  &  $7$ &   &  & $10$ & $10$ & $9$ \\
\hline
$x_3$ &   &   &   &  $3$ &  & $9$ & $8$ & $6$ \\
\hline
$x_4$ &   &   &   &  &  $5$ & $9$ & $7$ & $9$ \\
\hline
$O$ &  $5$ & $7$  &  $5$ &  $4$ & $4$ & $7$ &  & \\
\hline
$X$ &  $2$ & $3$  &  $5$ &  $3$ & $2$ &  & $3$ & \\
\hline
$Z$ &  $3$ & $3$  &  $4$ &  $1$ & $4$ &  &  & $3$ \\
\hline
\end{tabular}
\end{center}
\normalfont
where with $J(Y_1,Y_2)$ we denote the function $I (\widetilde{W}(Y_1),\widetilde{W}(Y_2))$.
Finally, from  formula (\ref{foM}) we get
\begin{align*}
\textbf{M}(x_0)=\mathbf{M_{\mathbb{O}}}(x_0)=\dfrac{1}{5} ( 3-10-5+7+1)-\dfrac{2}{5}+\dfrac{4}{5}=-\dfrac{2}{5}\\
\textbf{M}(x_1)=\mathbf{M_{\mathbb{O}}}(x_1)=\dfrac{1}{5} ( 7-12-7+7+1)-\dfrac{2}{5}+\dfrac{4}{5}=-\dfrac{2}{5} \\
\textbf{M}(x_2)=\mathbf{M_{\mathbb{O}}}(x_2)=\dfrac{1}{5} ( 7-10-5+7+1)-\dfrac{2}{5}+\dfrac{4}{5}=\dfrac{2}{5}\ \ \\
\textbf{M}(x_3)=\mathbf{M_{\mathbb{O}}}(x_3)=\dfrac{1}{5} ( 3-9-4+7+1)-\dfrac{2}{5}+\dfrac{4}{5}=0 \ \ \ \ \\
\textbf{M}(x_4)=\mathbf{M_{\mathbb{O}}}(x_4)=\dfrac{1}{5} ( 5-9-4+7+1)-\dfrac{2}{5}+\dfrac{4}{5}=\dfrac{2}{5}. \ \
\end{align*}
Observe that, since $O$ lies in the same cell for both $G_1$ and $G_2$, the Maslov degree is the same for generators of $C(G_1)$ and $C(G_2)$.

Similarly, we compute the Maslov degree with respect to $\mathbb{X}$ and to $\mathbb{Z}$ obtaining
\begin{align*}
& \mathbf{M_{\mathbb{X}}}(x_0)=\dfrac{1}{5} ( 3-7-2+3+1)-\dfrac{2}{5}+\dfrac{4}{5}=0\\
& \mathbf{M_{\mathbb{X}}}(x_1)=\dfrac{1}{5} ( 7-8-3+3+1)-\dfrac{2}{5}+\dfrac{4}{5}=\dfrac{2}{5}\\
& \mathbf{M_{\mathbb{X}}}(x_2)=\dfrac{1}{5} ( 7-10-5+3+1)-\dfrac{2}{5}+\dfrac{4}{5}=-\dfrac{2}{5}\\
& \mathbf{M_{\mathbb{X}}}(x_3)=\dfrac{1}{5} ( 3-8-3+3+1)-\dfrac{2}{5}+\dfrac{4}{5}=-\dfrac{2}{5}\\
& \mathbf{M_{\mathbb{X}}}(x_4)=\dfrac{1}{5} ( 5-7-2+3+1)-\dfrac{2}{5}+\dfrac{4}{5}=\dfrac{2}{5}\\
& \mathbf{M_{\mathbb{Z}}}(x_0)=\dfrac{1}{5} ( 3-8-3+3+1)-\dfrac{2}{5}+\dfrac{4}{5}=-\dfrac{2}{5}\\
& \mathbf{M_{\mathbb{Z}}}(x_1)=\dfrac{1}{5} ( 7-8-3+3+1)-\dfrac{2}{5}+\dfrac{4}{5}=\dfrac{2}{5}\\
& \mathbf{M_{\mathbb{Z}}}(x_2)=\dfrac{1}{5} ( 7-9-4+3+1)-\dfrac{2}{5}+\dfrac{4}{5}=0\\
& \mathbf{M_{\mathbb{Z}}}(x_3)=\dfrac{1}{5} ( 3-6-1+3+1)-\dfrac{2}{5}+\dfrac{4}{5}=\dfrac{2}{5}\\
& \mathbf{M_{\mathbb{Z}}}(x_4)=\dfrac{1}{5} ( 5-9-4+3+1)-\dfrac{2}{5}+\dfrac{4}{5}=-\dfrac{2}{5}.
\end{align*}

Now, using formula (\ref{foA}),  we can compute,  on one hand, the Alexander degree of $C(G_1)$ generators
\begin{align*}
& \mathbf{A} (x_0) = \dfrac{1}{2}( \mathbf{M_{\mathbb{O}}}(x_0) - \mathbf{M_{\mathbb{X}}}(x_0) -(n-1) )=\dfrac{1}{2} (-\dfrac{2}{5}-0-0)=-\dfrac{1}{5}\\
& \mathbf{A} (x_1) =\dfrac{1}{2} (-\dfrac{2}{5}-\dfrac{2}{5})=-\dfrac{2}{5}\\
& \mathbf{A} (x_2) =\dfrac{1}{2} (\dfrac{2}{5}+\dfrac{2}{5})=\dfrac{2}{5}\\
& \mathbf{A} (x_3) =\dfrac{1}{2} (0+\dfrac{2}{5})=\dfrac{1}{5}\\
& \mathbf{A} (x_4) =\dfrac{1}{2} (\dfrac{2}{5}-\dfrac{2}{5})=0,
\end{align*}
and, on the other hand, the  Alexander degree of $C(G_2)$ generators
\begin{align*}
& \mathbf{A} (x_0)=\dfrac{1}{2} (-\dfrac{2}{5}+\dfrac{2}{5}-0)=0\\
& \mathbf{A} (x_1) =\dfrac{1}{2} (-\dfrac{2}{5}-\dfrac{2}{5})=-\dfrac{2}{5}\\
& \mathbf{A} (x_2) =\dfrac{1}{2} (\dfrac{2}{5}+0)=\dfrac{1}{5}\\
& \mathbf{A} (x_3) =\dfrac{1}{2} (0-\dfrac{2}{5})=-\dfrac{1}{5}\\
& \mathbf{A} (x_4) =\dfrac{1}{2} (\dfrac{2}{5}+\dfrac{2}{5})=\dfrac{2}{5}.
\end{align*}

In the following tables we resume the  Maslov and Alexander degrees of  the generators.
\begin{center}
\begin{tabular}{|c|c|c|c|c|c|c|}
\hline 
 $C(G_1)$ & $x_0$ & $x_1$ & $x_2$ & $x_3$ & $x_4$ \\
\hline
\rule[-4mm]{0mm}{1cm} 
$\mathbf{M}$ & $-\dfrac{2}{5}$ &  $-\dfrac{2}{5}$ &  $\dfrac{2}{5}$ &  $0$ &  $\dfrac{2}{5}$  \\
\hline
\rule[-4mm]{0mm}{1cm}
 $\mathbf{A}$ & $-\dfrac{1}{5}$ &  $-\dfrac{2}{5}$ &  $\dfrac{2}{5}$ &  $\dfrac{1}{5}$  &  $0$ \\
 \hline
\end{tabular} 
\hspace{1cm}
\begin{tabular}{|c|c|c|c|c|c|c|}
\hline 
 $C(G_2)$ & $x_0$ & $x_1$ & $x_2$ & $x_3$ & $x_4$ \\
\hline
\rule[-4mm]{0mm}{1cm} 
$\mathbf{M}$ & $-\dfrac{2}{5}$ &  $-\dfrac{2}{5}$ &  $\dfrac{2}{5}$ &  $0$ &  $\dfrac{2}{5}$  \\
\hline
\rule[-4mm]{0mm}{1cm}
 $\mathbf{A}$ & $0$ &  $-\dfrac{2}{5}$ &  $\dfrac{1}{5}$ &  $-\dfrac{1}{5}$  &  $\dfrac{2}{5}$  \\
 \hline
\end{tabular}
\end{center}

After computing the three degrees for each generator, we look for pairs of generators connected by the boundary operator.
Since the  generators of $C(G_1)$ have different spin degree and,
by formula (\ref{fod}), the boundary operator $\partial$ preserves the spin degree,
 there is no  connection via boundary operator between the five generators. 
  This means that  
 each  generator of $C(G_1)$   is a generator of 
$H(C(G_1), \partial)$. Hence, by Proposition \ref{HFK}, we get

\begin{eqnarray*}
  \widehat{HFK} (K_1) \cong H(C(G_1), \partial) &\cong& \mathbb{Z}_2 [0,-\dfrac{2}{5} , -\dfrac{1}{5}] \oplus \mathbb{Z}_2 [1,-\dfrac{2}{5} , 
 -\dfrac{2}{5}] \oplus \mathbb{Z}_2 [2,\dfrac{2}{5} , \dfrac{2}{5}]\oplus \\ &\ & \oplus \mathbb{Z}_2 [3,0 , \dfrac{1}{5}] \oplus \mathbb{Z}_2 [4,\dfrac{2}{5} , 0],
\end{eqnarray*}
where $\mathbb{Z}_2 [i, j, k]$ denotes a $\mathbb{Z}_2$-vector space embedded with spin degree $i$, Maslov degree $j$ and Alexander degree $k$.

Similarly we get 
\begin{eqnarray*}
 \widehat{HFK} (K_1)\cong H(C(G_2), \partial) &\cong& \mathbb{Z}_2 [0,-\dfrac{2}{5} , 0] \oplus \mathbb{Z}_2 [1,-\dfrac{2}{5} , 
 -\dfrac{2}{5}] \oplus \mathbb{Z}_2 [2,\dfrac{2}{5} , \dfrac{1}{5}] \oplus\\ &\ & \oplus \mathbb{Z}_2 [3,0 , -\dfrac{1}{5}] \oplus 
 \mathbb{Z}_2 [4,\dfrac{2}{5} , \dfrac{2}{5}].
\end{eqnarray*}

By a straightforward computation we obtain 
\begin{eqnarray*}
 \widehat{HFK} (-K_1) &\cong& \mathbb{Z}_2 [0,-\dfrac{2}{5} , -\dfrac{1}{5}] \oplus \mathbb{Z}_2 [1,\dfrac{2}{5} , 0]
\oplus \mathbb{Z}_2 [2,0,\dfrac{1}{5}] \oplus\\ &\ &\oplus \mathbb{Z}_2 [3 ,\dfrac{2}{5}, \dfrac{2}{5}] 
 \oplus  \mathbb{Z}_2 [4,-\dfrac{2}{5} , -\dfrac{2}{5}],
\end{eqnarray*}

as predicted by Proposition \ref{HFLorient}.
 
\paragraph{Computation of example \ref{esHFL}}
We compute the Link Floer Homology of the links $L_A,L_B\subset L(4,1)$ of Figure  \ref{CE2grid}. Using the grid diagram 
with grid number one depicted in the figure, the computations for the knot $L_A$
are really similar to the ones done in the previous example. We get 
\begin{eqnarray*}
 \widehat{HFL} (L_A) &\cong&  
 \mathbb{Z}_2 [0,\dfrac{1}{2},\dfrac{1}{2}] \oplus \mathbb{Z}_2 [1,\dfrac{1}{2},-\dfrac{1}{2}] \oplus \mathbb{Z}_2 [2,-\dfrac{1}{2},\dfrac{1}{2}] \oplus \\ 
&\ & \oplus \mathbb{Z}_2 [3,-\dfrac{1}{2},-\dfrac{1}{2}].
\end{eqnarray*}
Much more work is necessary to compute the Link Floer Homology of the two components link $L_B$.
\begin {figure}[htb]
\begin{center}
\includegraphics[width=6in]{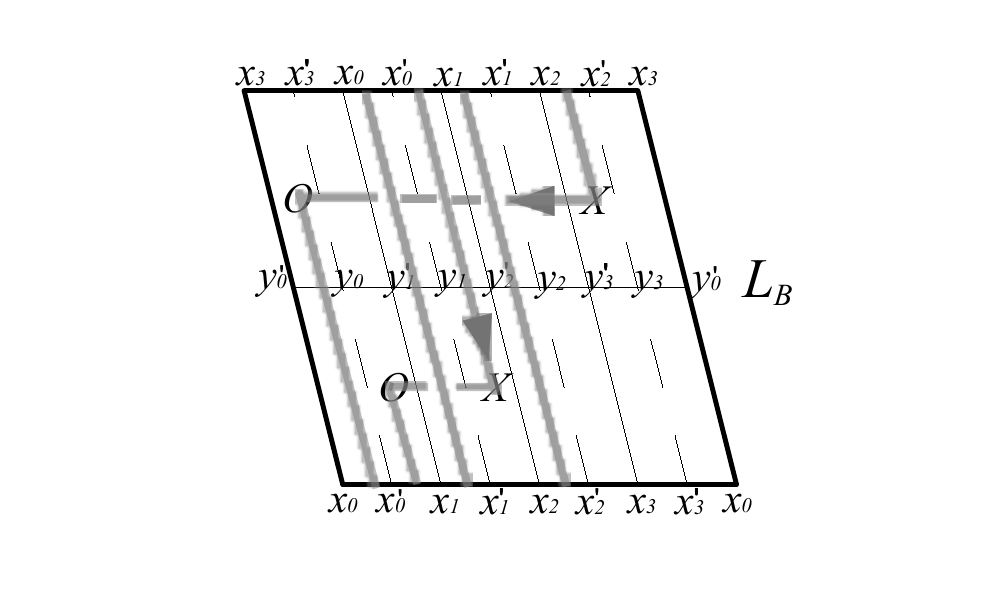}\label{figfinale}
\caption {\textnormal{Grid diagram $G_B$ for the link $L_B\subset L(4,1)$.}}
\end{center}
\end{figure}
Referring to Figure 20,  we identify the set of the  generators of $C(G_B)$ with
the set composed by pairs $x_i y_j$ and by pairs $x_i \textquoteright y_j \textquoteright$ for $i,j \in \{1, \ldots , p \}$.
By formula (\ref{foS}), we get  
$$\mathbf{S} (x_i y_j) \equiv i+j \equiv \mathbf{S} (x_i \textquoteright y_j \textquoteright)  \mod p.$$

Implementing the algorithm described in the first paragraph of the Appendix, we can use a calculator to compute both the Maslov and the Alexander degrees of the generators of $C(G_B)$. The results
are contained in the  four tables starting at  page \pageref{tavole}.

Once we have computed all the three degrees, we want to establish which generators are connected by the boundary operator.
From formula (\ref{fod}), 
two generators  of $C(G_B)$, may be connected by the boundary operator if they have the same spin and Alexander degrees and if their Maslov degrees differ by 1. 
Given two such generators  $\mathbf{x}$ and $\mathbf{y}$, 
consider the set $N_{\mathbf{x},\mathbf{y}}$ of admissible parallelograms connecting them and  containing neither
$O$ nor $X$. We have that the boundary operator connects $\mathbf{x}$ and $\mathbf{y}$ if and only if 
$\# (N_{\mathbf{x},\mathbf{y}}) \negthickspace \equiv 1 \mod 2$.

\begin{landscape}
~
\vfill

\label{tavole}
\begin{table}[h]
\begin{center}
\fontfamily{phv} \selectfont
\begin{tabular}{|c|c|c|c|c|c|c|c|c|c|}
\hline 
 & A$_0$ & B$_0$ & C$_0$ & D$_0$ & E$_0$ & F$_0$ & G$_0$ & H$_0$ \\
\hline
 $\mathbf{S} =0$ & $x_0 y_0$ & $x_0 \textquoteright  y_0 \textquoteright$  & $x_1 y_3$ & $x_1 \textquoteright  y_3 \textquoteright$  & $x_2 y_2$ & $x_2 \textquoteright  y_2 \textquoteright$  & $x_3 y_1$ & $x_3 \textquoteright  y_1 \textquoteright$  \\
\hline
\rule[-4mm]{0mm}{1cm} 
$\mathbf{M}$ & $-\dfrac{3}{4}$ &  $-\dfrac{7}{4}$ &  $-\dfrac{3}{4}$ &  $\dfrac{1}{4}$ &  $\dfrac{5}{4}$ &  $\dfrac{1}{4}$  &  $-\dfrac{3}{4}$ &  $-\dfrac{7}{4}$ \\
\hline
\rule[-4mm]{0mm}{1cm}
 $\mathbf{A}$ & $(-\dfrac{7}{8} , - \dfrac{7}{8})$ & $(-\dfrac{7}{8} , - \dfrac{7}{8})$ & $(-\dfrac{7}{8} , \dfrac{1}{8})$ & $(\dfrac{1}{8} , \dfrac{1}{8})$  & $(\dfrac{1}{8} , \dfrac{1}{8})$  & $(\dfrac{1}{8} , \dfrac{1}{8})$ & $(\dfrac{1}{8} , - \dfrac{7}{8})$ & $(-\dfrac{7}{8} , - \dfrac{7}{8})$  \\
\hline
\end{tabular}
\end{center}
\caption{Generators of $C(G_B)$  having spin degree $0$.}
\end{table}
\normalfont

\vfill
\begin{table}[h]
\begin{center}
\fontfamily{phv} \selectfont
\begin{tabular}{|c|c|c|c|c|c|c|c|c|c|}
\hline 
 & A$_1$ & B$_1$ & C$_1$ & D$_1$ & E$_1$ & F$_1$ & G$_1$ & H$_1$ \\
\hline
 $\mathbf{S} =1$ & $x_0 y_1$ & $x_0 \textquoteright  y_1 \textquoteright$  & $x_1 y_0$ & $x_1 \textquoteright  y_0 \textquoteright$  & $x_2 y_3$ & $x_2 \textquoteright  y_3 \textquoteright$  & $x_3 y_2$ & $x_3 \textquoteright  y_2 \textquoteright$  \\
\hline 
$\mathbf{M}$ & $-2$ & $-1$ & $0$ & $-1$ & $0$ & $1$  & $0$ & $-1$ \\
\hline
\rule[-4mm]{0mm}{1cm}
 $\mathbf{A}$ &  $(-\dfrac{5}{8} , -\dfrac{9}{8})$ &  $(-\dfrac{5}{8} , -\dfrac{9}{8})$ &  $(-\dfrac{5}{8} , -\dfrac{1}{8})$ &  $ (-\dfrac{5}{8} , -\dfrac{1}{8})$  & $(-\dfrac{5}{8} , -\dfrac{1}{8})$  &  $(\dfrac{3}{8} , -\dfrac{1}{8})$ &  $(\dfrac{3}{8} , -\dfrac{1}{8})$ &  $(-\dfrac{5}{8} , -\dfrac{1}{8})$  \\
\hline
\end{tabular}
\end{center}
\caption{Generators of $C(G_B)$  having spin degree $1$.}
\end{table}
\normalfont
\vfill

\begin{table}[h]
\begin{center}
\fontfamily{phv} \selectfont
\begin{tabular}{|c|c|c|c|c|c|c|c|c|c|}
\hline 
 & A$_2$ & B$_2$ & C$_2$ & D$_2$ & E$_2$ & F$_2$ & G$_2$ & H$_2$  \\
\hline
 $\mathbf{S} =2$ & $x_0 y_2$ & $x_0 \textquoteright  y_2 \textquoteright$  & $x_1 y_1$ & $x_1 \textquoteright  y_1 \textquoteright$  & $x_2 y_0$ & $x_2 \textquoteright  y_0 \textquoteright$  & $x_3 y_3$ & $x_3 \textquoteright  y_3 \textquoteright$  \\
\hline 
\rule[-4mm]{0mm}{1cm}
$\mathbf{M}$ & $-\dfrac{7}{4}$ &  $-\dfrac{3}{4}$ &  $\dfrac{1}{4}$ &  $-\dfrac{3}{4}$ &  $\dfrac{1}{4}$ &  $-\dfrac{3}{4}$  &  $\dfrac{1}{4}$ &  $-\dfrac{3}{4}$ \\
\hline
\rule[-4mm]{0mm}{1cm}
 $\mathbf{A}$ & $(-\dfrac{3}{8} , -\dfrac{3}{8})$ &  $(-\dfrac{3}{8} , -\dfrac{3}{8})$ &  $(-\dfrac{3}{8} , -\dfrac{3}{8})$ &  $(-\dfrac{3}{8} , -\dfrac{3}{8})$  & $(-\dfrac{3}{8} , -\dfrac{3}{8})$ &  $(-\dfrac{3}{8} , -\dfrac{3}{8})$ &  $(-\dfrac{3}{8} , -\dfrac{3}{8})$ &  $(-\dfrac{3}{8} , -\dfrac{3}{8})$  \\
\hline
\end{tabular}
\end{center}
\caption{Generators of $C(G_B)$  having spin degree $2$.}
\end{table}
\normalfont
\vfill

\begin{table}[h]
\begin{center}
\fontfamily{phv} \selectfont
\begin{tabular}{|c|c|c|c|c|c|c|c|c|c|}
\hline 
 & A$_3$ & B$_3$ & C$_3$ & D$_3$ & E$_3$ & F$_3$ & G$_3$ & H$_3$ \\
\hline
 $\mathbf{S} =3$ & $x_0 y_3$ & $x_0 \textquoteright  y_3 \textquoteright$  & $x_1 y_2$ & $x_1 \textquoteright  y_2 \textquoteright$  & $x_2 y_1$ & $x_2 \textquoteright  y_1 \textquoteright$  & $x_3 y_0$ & $x_3 \textquoteright  y_0 \textquoteright$  \\
\hline 
$\mathbf{M}$ & $-2$ & $-1$ & $0$ & $1$ & $0$ & $-1$  & $0$ & $-1$ \\
\hline
\rule[-4mm]{0mm}{1cm}
 $\mathbf{A}$ & $(-\dfrac{9}{8} , -\dfrac{5}{8})$ &  $(-\dfrac{1}{8} , -\dfrac{5}{8})$ &  $(-\dfrac{1}{8} , \dfrac{3}{8})$ &  $(-\dfrac{1}{8} , \dfrac{3}{8})$  & $(-\dfrac{1}{8} , -\dfrac{5}{8})$ &  $(-\dfrac{1}{8} , -\dfrac{5}{8})$ &  $(-\dfrac{1}{8} , -\dfrac{5}{8})$ &  $(-\dfrac{9}{8} , -\dfrac{5}{8})$  \\
\hline
\end{tabular}
\end{center}
\caption{ Generators of $C(G_B)$   having spin degree $3$.}
\end{table}
\normalfont
\vfill
~
\end{landscape}

In order to compute $H(C(G_B),\partial)$ we study four different cases, depending on the spin degree of the generators analyzed.
Moreover we describe the behavior of the boundary operator using a diagram,  with the following conventions  
\begin{itemize}
 \item  if $\partial(Z)=0$,  no arrow will start from the generator $Z$  of $C(G_B)$;
 \item if $\partial(Z)=Z_1+\cdots+Z_k$ there will be an arrow starting from $Z$ and ending in $Z_i$, for $i=1,\ldots,k$.
\end{itemize}

Let us start from spin degree $0$.  We have\

\begin {figure}[h!]
\begin{center}
 \includegraphics[width=4in]{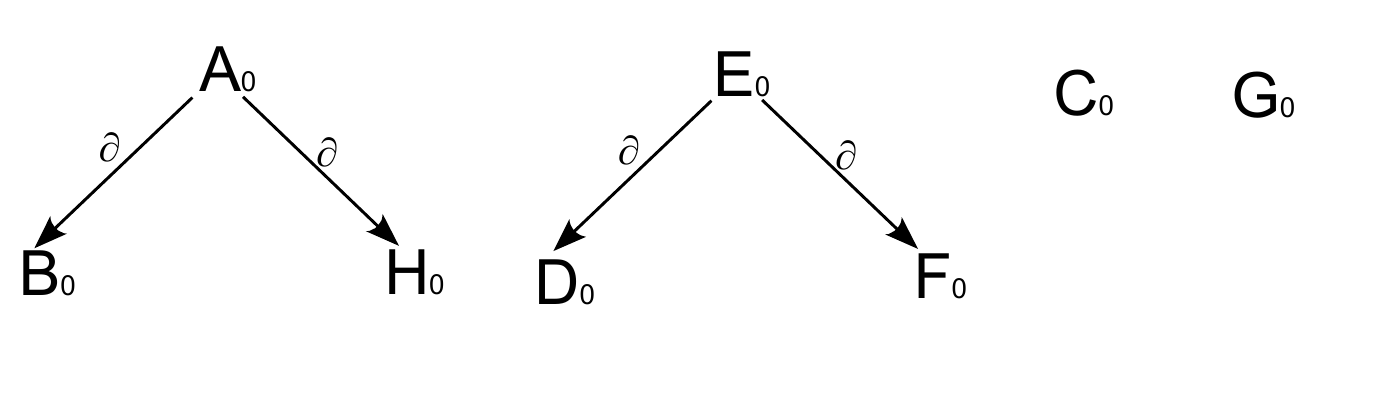}
\end{center}
\end{figure} \vspace{0.01in}

so {\fontfamily{phv} \selectfont B$_0$, H$_0$, D$_0$, F$_0$, C$_0$} and 
{\fontfamily{phv} \selectfont G$_0$} are cycles and are not the boundary of any chain. The cycle  
{\fontfamily{phv} \selectfont B$_0$} is equivalent to {\fontfamily{phv} \selectfont H$_0$},  since their sum is  the boundary of 
{\fontfamily{phv} \selectfont A$_0$}, and, the same holds for  {\fontfamily{phv} \selectfont D$_0$} and {\fontfamily{phv} \selectfont F$_0$}.
Then  {\fontfamily{phv} \selectfont B$_0$, C$_0$, G$_0$} and  {\fontfamily{phv} \selectfont D$_0$} {\fontfamily{phv}} are the four generators of
$H(C(G_B), \partial)$ having  spin degree equal to $0$.

For spin degree $1$ we have\\
\begin {figure}[h!]
\begin{center}
 \includegraphics[width=3in]{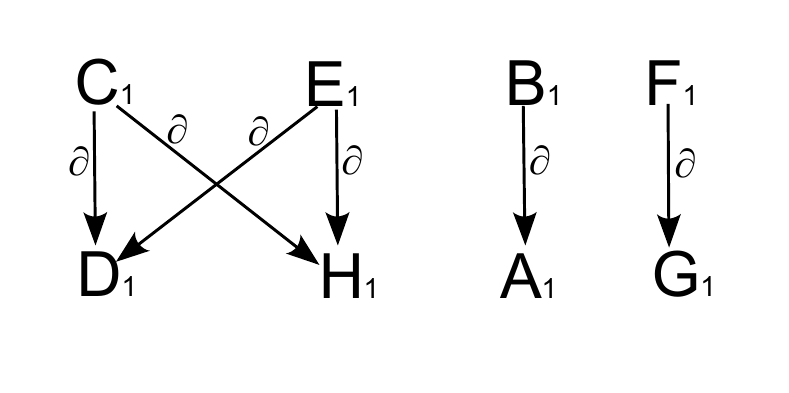}
\end{center}
\end{figure}

so  {\fontfamily{phv} \selectfont D$_1$} (which is equivalent to {\fontfamily{phv} \selectfont H$_1$}) and {\fontfamily{phv} \selectfont C$_1+$E$_1$} are the
two generators of $H(C(G_B), \partial)$ having spin degree  $1$.

For spin degree $2$ we get\\

\begin {figure}[h!]
\begin{center}
 \includegraphics[width=3in]{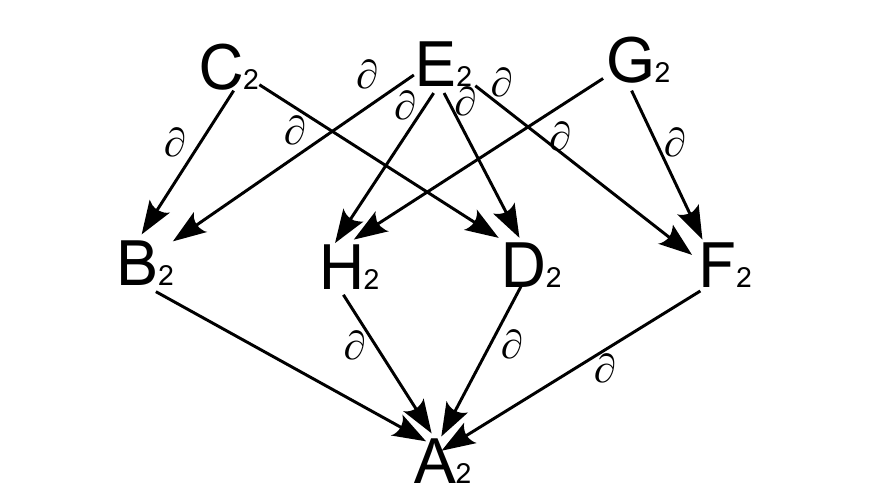}
\end{center}
\end{figure}
so  {\fontfamily{phv} \selectfont C$_2+$E$_2+$G$_2$}, {\fontfamily{phv} \selectfont B$_2+$H$_2$}, {\fontfamily{phv}
\selectfont B$_2+$F$_2$} and {\fontfamily{phv} \selectfont H$_2+$D$_2$} are the four generators of $H(C(G_B), \partial)$ having spin degree $2$.

Finally, for spin degree $3$ we get\\

\begin {figure}[h!]
\begin{center}
 \includegraphics[width=3in]{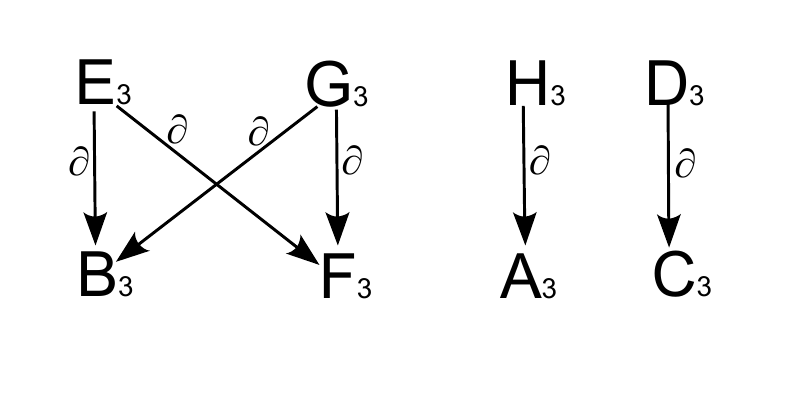}
\end{center}
\end{figure}
so  {\fontfamily{phv} \selectfont B$_3$} (which is equivalent to  {\fontfamily{phv} \selectfont F$_3$}) and {\fontfamily{phv} \selectfont E$_3+$G$_3$}
are the two generators of $H(C(G_B), \partial)$ having spin degree $3$.\\

According to Proposition \ref{proplorenzo1}, we have  $\widehat{HFL} (L_B)  \cong H(C(G_B) , \partial) $ and so
\begin{eqnarray*}
\widehat{HFL} (L_B)  &\cong& \mathbb{Z}_2 \left[0,\dfrac{1}{4} , 
 (\dfrac{1}{8},\dfrac{1}{8} ) \right] \oplus\mathbb{Z}_2 \left[0,-\dfrac{3}{4} , (-\dfrac{7}{8},\dfrac{1}{8} ) \right]  
\oplus \mathbb{Z}_2 \left[0,-\dfrac{3}{4} , (\dfrac{1}{8},-\dfrac{7}{8} ) \right] \oplus\\ &\ & \mathbb{Z}_2 \left[0, -\dfrac{7}{4} ,
(-\dfrac{7}{8},-\dfrac{7}{8}) \right] \oplus\mathbb{Z}_2 \left[1, 0 , (-\dfrac{5}{8},-\dfrac{1}{8}) \right] 
 \oplus  \mathbb{Z}_2 \left[1, -1 , (-\dfrac{5}{8},-\dfrac{1}{8}) \right] \oplus\\ &\ & \oplus \mathbb{Z}_2 \left[2, \dfrac{1}{4} , 
 (-\dfrac{3}{8},-\dfrac{3}{8}) \right] \oplus {\mathbb{Z}_2 \left[2, -\dfrac{3}{4} , (-\dfrac{3}{8},-\dfrac{3}{8}) \right]}^3 \oplus\\
&\ &\oplus \mathbb{Z}_2 \left[3, 0 , (-\dfrac{1}{8},-\dfrac{5}{8}) \right] \oplus \mathbb{Z}_2 \left[3,-1 , (-\dfrac{1}{8},-\dfrac{5}{8}) \right],
\end{eqnarray*}
where $\mathbb{Z}_2 [i, j , (k_1,k_2)]$ indicates a $\mathbb{Z}_2$-vector space generated by an element of spin degree $i$, Maslov degree $j$ and Alexander bigrading $(k_1,k_2)$.\\

\end{section}

\vspace{2mm}
\textit{Acknowledgments:} The authors are grateful to Chris
Cornwell and Ken Baker for helpful discussions and to Ely Grigsby for a program computing HFK of knots in lens spaces with grid diagram at most two.


\vspace{15 pt} {ALESSIA CATTABRIGA, Department of Mathematics,
University of Bologna, ITALY. E-mail: alessia.cattabriga@unibo.it}

\vspace{15 pt} {ENRICO MANFREDI, Department of Mathematics,
University of Bologna, ITALY. E-mail: enrico.manfredi3@unibo.it}

\vspace{15 pt} {LORENZO RIGOLLI, Department of Mathematics,
Ruhr University of Bochum, GERMANY. E-mail: Lorenzo.Rigolli@rub.de}


\end{document}